\documentclass[a4paper, 12pt, oneside, reqno]{amsart}
\calclayout

\usepackage[utf8]{inputenc}
\usepackage[T1]{fontenc}
\usepackage[english]{babel}
\usepackage{amsmath, amsfonts, amssymb, amsthm, mathtools, indentfirst, bbm}
\usepackage[dvipsnames]{xcolor}
\usepackage[shortlabels]{enumitem}
\usepackage[languagenames,fixlanguage]{babelbib}
\usepackage{commath}
\usepackage{textcomp}
\usepackage{mathrsfs}
\usepackage{todonotes}
\usepackage{pgfplots}
\usepackage{subcaption}
\usepackage{graphicx}
\allowdisplaybreaks

\usepackage[unicode=true, pdfstartview={XYZ null null 1.00}, colorlinks=true, linkcolor=blue, urlcolor=blue, citecolor=purple]{hyperref}

\newcommand{\bbN}{\mathbb{N}}

\newcommand{\bbR}{\mathbb{R}}

\newcommand{\bbP}{\mathbb{P}}
\newcommand{\bbE}{\mathbb{E}}

\newcommand{\aps}[1]{\vert #1 \vert}
\newcommand{\APS}[1]{\left\vert #1 \right\vert}
\newcommand{\floor}[1]{\lfloor #1 \rfloor}

\newcommand{\OBL}[1]{\left( #1 \right)}
\newcommand{\UGL}[1]{\left[ #1 \right]}
\newcommand{\VIT}[1]{\left\{ #1 \right\}}

\newcommand{\FJADEF}[3]{{#1}:{#2}\to{#3}}

\let\ge\geqslant
\let\le\leqslant

\newtheorem{DEF}{Definition}[section]
\newtheorem{TM}[DEF]{Theorem}
\newtheorem{PROP}[DEF]{Proposition}
\newtheorem{LM}[DEF]{Lemma}
\newtheorem{COR}[DEF]{Corollary}
\newtheorem{REM}[DEF]{Remark}
\newtheorem{CON}[DEF]{Conjecture}
\newtheorem{QUE}[DEF]{Question}

\numberwithin{equation}{section}

\makeatletter
\def\@tocline#1#2#3#4#5#6#7{\relax
  \ifnum #1>\c@tocdepth 
  \else
    \par \addpenalty\@secpenalty\addvspace{#2}%
    \begingroup \hyphenpenalty\@M
    \@ifempty{#4}{%
      \@tempdima\csname r@tocindent\number#1\endcsname\relax
    }{%
      \@tempdima#4\relax
    }%
    \parindent\z@ \leftskip#3\relax \advance\leftskip\@tempdima\relax
    \rightskip\@pnumwidth plus4em \parfillskip-\@pnumwidth
    #5\leavevmode\hskip-\@tempdima
      \ifcase #1
       \or\or \hskip 1em \or \hskip 2em \else \hskip 3em \fi%
      #6\nobreak\relax
    \dotfill\hbox to\@pnumwidth{\@tocpagenum{#7}}\par
    \nobreak
    \endgroup
  \fi}
\makeatother

\title[Asymptotic fairness in  voting]{On asymptotic fairness in voting with greedy sampling}
	\author{Abraham Gutierrez}
		\address{
		Abraham Gutierrez,
		Institute of Discrete Mathematics
		Graz University of Technology, Graz, Austria}
		\email{a.gutierrez@math.tugraz.at}
	\author{Sebastian M\"uller}
		\address{
		Sebastian M\"uller,
		Aix Marseille Universit\'e, CNRS, Centrale Marseille, I2M - UMR 7373, 13453 Marseille, France \&
		IOTA Foundation, 10405 Berlin, Germany
		}
		\email{sebastian.muller@univ-amu.fr}
	\author{Stjepan Šebek}
		\address{
		Stjepan Šebek,
		Department of Applied Mathematics,
		Faculty of Electrical Engineering and Computing, University of Zagreb, Croatia
		\&
		Institute of Discrete Mathematics
		Graz University of Technology, Graz, Austria}
		\email{stjepan.sebek@fer.hr}
\date{}

\begin{document}

\date{\today}

\begin{abstract}
The basic idea of voting protocols is that nodes query a sample of other nodes and adjust their own opinion throughout several rounds based on the proportion of the sampled opinions.
In the classic model, it is assumed that all nodes have the same weight. We study voting protocols for heterogeneous weights with respect to fairness. A voting protocol is \emph{fair} if the influence on the eventual outcome of a given participant is linear in its weight. Previous work used sampling with replacement to construct a fair voting scheme. However, it was shown that using greedy sampling, i.e., sampling with replacement until a given number of distinct elements is chosen, turns out to be more robust and performant. 

In this paper, we study fairness of voting protocols with greedy sampling and propose a voting scheme that is asymptotically fair for a broad class of weight distributions. We complement our theoretical findings with numerical results and present several open questions and conjectures. 
\end{abstract}

\subjclass[2010]{68M14,  94A20, 91A20}
\keywords{asymptotic fairness, consensus protocol, voting scheme, heterogeneous network, Sybil protection}

\maketitle

\section{Introduction}\label{sec:Introduction}
This article focuses on fairness in binary voting protocols.  Marquis de Condorcet observed the principle of voting in 1785 \cite{marquis}. Let us suppose there is a large population of voters, and each of them  independently votes ``correctly'' with probability $p > 1/2$. Then, the probability that the outcome of a majority vote is ``correct''  grows with the sample size and converges to one. In many applications, for instance, distributed computing, it is not feasible that every node queries every other participant and a centralized entity that collects the votes of every participant and communicates the final result is not desired. Natural decentralized solutions with low message complexity are the so-called voting consensus protocols.  Nodes query other nodes (only a sample of the entire population) about their current opinion and adjust their own opinion throughout several rounds based on the proportion of other opinions they have observed.

These protocols may achieve good performances in noiseless and undisturbed networks. However, their performances significantly decreases with noise \cite{GaKuLe:78,KaMo:07} or errors \cite{MoDiAm:04} and may completely fail in a Byzantine setting \cite{fpcsim}. Recently, \cite{fpc} introduced a variant of the standard voting protocol, the so-called fast probabilistic consensus (FPC), that is robust in Byzantine environment.  The performance of FPC was then studied using Monte-Carlo simulations in \cite{fpcsim}.   The above voting protocols are tailored for homogeneous networks where all votes have equal weight. In \cite{MPCS20, manaFPC} FPC was generalized to heterogeneous settings. These studies also revealed that how votes are sampled does have a considerable impact on the quality of the protocol.  

In a weighted or unweighted sampling, there are three different ways to choose a sample from a population:
\begin{enumerate}
\item choose with replacement until one has $m\in \bbN$ elements;
\item choose with replacement until one has $k\in \bbN$ distinct elements;
\item choose without replacement until one has $k=m$ (distinct) elements.
\end{enumerate}
The first method is usually referred to as sampling with replacement. While in the $1950$s, e.g., \cite{RaKh:58}, the second way was called sampling without replacement, sampling without replacement nowadays usually refers to the third possibility. To avoid any further confusion, we call in this paper the second possibility greedy sampling. 


Most voting protocols assume that every participant has the same weight. In heterogeneous situations, this does not reflect possible differences in weight or influence of the participants. An essential way in which weights improve voting protocols is by securing that the voting protocol is fair in the sense that the influence of a node on another node's opinion is proportional to its weight. This fairness is an essential feature of a voting protocol both for technical reasons, e.g., defense against Sybil attacks, and social reasons, e.g., participants may decide to leave the network if the voting protocol is unfair. Moreover, an unfair situation may incentivize participants to split their weight among several participants or increase their weight by pooling with other participants. These incentives may lead to undesired effects as fragility against Sybil attacks and centralization.

The construction of a fair voting consensus protocols with weights was recently discussed in \cite{MPCS20, manaFPC}. We consider a network with $N$ nodes (or participants), identified with the integers $\{1, \ldots, N\}$. The weights of the nodes are described by $(m_i)_{i \in \bbN}$ with $\sum_{i = 1}^{N} m_i = 1$, $m_i \ge 0$ being the weight of the node $i$. Every node $i$ has an initial state or opinion $s_{i}\in \{0, 1\}$. Then, at each (discrete) time step, each node chooses $k \in \bbN$ random nodes from the network and queries their opinions. This sampling can be done in one of the three ways described above. For instance, \cite{MPCS20} studied fairness in the case of sampling with replacement. The mathematical treatment of this case is the easiest of the three possibilities. However, simulations in \cite{manaFPC} strongly suggest that the performance of some consensus protocols are considerably better in the case of greedy sampling. The main object of our work is the mathematical analysis of weighted greedy sampling with respect to fairness. 

The weights of the node may enter at two points during the voting: in sampling and in weighting the collected votes or opinions. We consider a first weighting function $\FJADEF{f}{[0, \infty)}{[0, \infty)}$ that describes the weight of a node in the sampling. More precisely, a node $i$ is chosen with probability
\begin{equation}\label{eq:prob_from_mana}
	p_i := \frac{f(m_i)}{\sum_{j = 1}^{N} f(m_j)}.
\end{equation}
We call this function $f$  the sampling weight function. A natural weight function is $f\equiv id$; a node is chosen proportional to its weight.

As discussed later in the paper, we are interested in how the weights influence the voting if the number of nodes in the network tends to infinity. Therefore, we often consider the situation with an infinite number of nodes. The weights of these nodes are again  described by $(m_i)_{i \in \bbN}$ with $\sum_{i = 1}^{\infty} m_i = 1$. A network of $N$ nodes is then described by setting $m_i = 0$ for all $i > N$.

Once a node has chosen $k$ distinct elements, by greedy sampling,  it calculates a weighted mean opinion of these nodes. Let us denote by $S_{i}$ the multi-set of the sample for a given node $i$. The mean opinion of the sampled node is 
\begin{equation}\label{eq:meanOpinion}
\eta_{i} := \frac{\sum_{j \in S_{i}} g(m_{j}) s_{j}}{\sum_{j \in S_{i}} g(m_{j})},
\end{equation}
 where $\FJADEF{g}{[0, \infty)}{[0, \infty)}$ is a second weight function that we dub the  averaging weight function.  The pair $(f,g)$ of the two weight functions is called a voting scheme. 
 
In standard majority voting every  node adjusts its opinion as follows: if $\eta_{i}<1/2$ it updates its own opinion $s_{i}$ to $0$ and  if $\eta_{i}>1/2$ to $1$. The case of a draw, $\eta_{i}=1/2$, may be solved by randomization or choosing deterministically one of the options. After the opinion update, every node would re-sample and continue this procedure until some stopping condition is verified. In general, such a protocol aims that all nodes finally agree on one opinion or, in other words, find consensus.  As mentioned above, this kind of protocol works well in a non-faulty environment. However, it fails to reach consensus when some nodes do not follow the rules or even try to hinder the other nodes from reaching consensus. In this case, one speaks of honest nodes, the nodes which follow the protocol, and malicious nodes, the nodes that try to interfere. 
An additional feature was introduced by \cite{fpc} that makes this kind of consensus protocol robust to some given proportion of malicious nodes in the network.  

Let us briefly explain this crucial feature.  As in \cite{fpcsim, MPCS20, manaFPC} we consider a basic version of the FPC introduced in \cite{fpc}.  Let $U_{t}$, $t=1, 2,\ldots$ be i.i.d.~random variables with law $\mathrm{Unif}( [\beta, 1-\beta])$ for some parameter $\beta \in [0,1/2]$.  Every node $i$ has an opinion or state. We note $s_{i}(t)$ for the opinion of the node $i$ at time $t$. Opinions take values in $\{0,1\}$. Every node $i$ has an initial opinion $s_{i}(0)$.
The update rules for the opinion of a node $i$ is then given by
\begin{equation*}
s_{i}(1)=\left\{ \begin{array}{ll}
1, \mbox{ if } \eta_{i}(1) \geq \tau, \\
0, \mbox{ otherwise,}
\end{array}\right. 
\end{equation*}
for some $\tau \in [0,1]$.
For $t\geq 1$:
\begin{equation*}
s_{i}(t+1)=\left\{ \begin{array}{ll}
1, \mbox{ if } \eta_{i}(t+1) > U_{t}, \\
0, \mbox{ if } \eta_{i}(t+1) < U_{t}, \\
s_{i}(t), \mbox{ otherwise.}
\end{array}\right. 
\end{equation*}
Note that if $\tau=\beta=0.5$, FPC reduces to a standard majority consensus. It is important that the above sequence of random variables $U_t$ are the same for all nodes. The randomness of the threshold effectively reduces the capabilities of an attacker to control the opinions of honest nodes and it also increases the rate of convergence in the case of honest nodes only. 
Since in this paper we focus our attention mainly on the construction and analysis of the voting schemes $(f,g)$ we refer to 
\cite{fpcsim, MPCS20, manaFPC} for more details on FPC.

We concentrate mostly on the case $f \equiv id$ and $g \equiv 1$. For the voting scheme with sampling with replacement, it was shown in \cite[Theorem 1]{MPCS20} that  for $g \equiv 1$, i.e., when the opinions of different nodes are not additionally weighted after the nodes are sampled, the voting scheme $(f, g)$ is fair, see Definition \ref{def:fairness}, if and only if $f \equiv id$.  For $f \equiv id$, the probability of sampling a node $j$ satisfies $p_j = m_j$ because we assumed that $\sum_{i = 1}^{\infty}m_i = 1$. In many places we use $m_j$ and $p_j$ interchangeably, and both notations refer simultaneously to the weight of the node $j$ and the probability that the node $j$ is sampled.

Our primary goal is to verify whether the voting scheme $(id, 1)$ is fair in the case of greedy sampling.  We show in  Proposition \ref{prop:unfairness} that the voting scheme $(id, 1)$ is in general not fair.  For this reason, we introduce the notion of asymptotic fairness, see Definition \ref{def:asymptotic_fairness}. Even though the definition of asymptotic fairness is very general, the best example to keep in mind is when the number of nodes grows to infinity. An important question related to the robustness of the protocol against Sybil attacks is if the gain in influence on the voting obtained by splitting one node in ``infinitely'' many nodes is limited. 
	
We find a sufficient condition on the sequence of weight distributions $\{(m_i^{(n)})_{i \in \bbN}\}_{n \in \bbN}$ for asymptotic fairness, see Theorem \ref{thm:asymptotic_fairness}. In particular, this ensures robustness against Sybil attacks for wide classes of weight distributions. However, we also note that there are situations that are not asymptotically fair,  see Corollary \ref{cor:asymp_unfairness} and Remark \ref{rem:asymp_unfairness}.

A key ingredient of our proof is a preliminary result on greedy sampling. This is  a generalization of some of the results of \cite{RaKh:58}. More precisely, we obtain a formula for the joint distribution of the random vector $(A_k^{}(i), v_k^{})$. Here, the random variable $v_k^{}$, defined in \eqref{eq:def_of_vkP}, counts the number of samplings needed to sample $k$ different elements, and the random variable $A_k^{}(i)$, defined in \eqref{eq:def_of_AkP}, counts how many times in those $v_k^{}$ samplings, the node $i$ was sampled. The result of asymptotic fairness, Corollary \ref{cor:asymp_unfairness}, relies on a stochastic coupling that compares the nodes' influence before and after splitting. We use this coupling also in the simulations in Section \ref{sec:Simulations};  it considerably improves the convergence of our simulations by reducing the variance. 

Fairness plays a prominent role in many areas of science and applications. It is, therefore, not astonishing that it plays its part also in distributed ledger technologies. For instance, proof-of-work in Nakamoto consensus ensures that the probability of creating a new block is proportional to the computational power of a node; see \cite{ChPaCr:19} for an axiomatic approach to block rewards and further references. In proof-of-stake blockchains, the probability of creating a new block is usually proportional to the node's balance. However, this does not always have to be the optimal choice, \cite{LeRePi:20, PopovNxt}. 

Our initial motivation for this paper was to show that the consensus protocol used in the next generation protocol of IOTA, see \cite{coordicide}, is robust against splitting and merging. Both effects are not desirable in a decentralized and permissionless distributed system. We refer to \cite{MPCS20, manaFPC} for more details.  Besides this, we believe that the study of the different voting schemes is of theoretical interest and that many natural questions are still open, see Section \ref{sec:Simulations}.

We organize the article as follows. Section \ref{sec:Preliminaries} defines the key concepts of this paper: voting power, fairness, and asymptotic fairness. We also recall  Zipf's law that we use to model the weight distribution of the nodes. Even though our results are obtained in a  general setting, we discuss in several places how these results apply to the case of Zipf's law, see Subsection \ref{subsec:Zipf_law} and Figure \ref{fig:Zipf_law}. Section \ref{sec:Dist_of_sample_size} is devoted to studying greedy sampling on its own.  We find the joint probability distribution of sample size and occurrences of the nodes,  $(A_k^{}(i), v_k^{})$,  and develop several asymptotic results we use in the rest of the paper. In Section \ref{sec:Asymptotic_fairness} we show that the voting scheme $(id, 1)$ is in general not fair. However, we give a sufficient condition on the sequence of weight distributions that ensures asymptotic fairness. We provide an example where, without this condition, the voting scheme $(id, 1)$ is not asymptotically fair. Section \ref{sec:Simulations} contains a short simulation study. Besides illustrating the theoretical results developed in the paper, we investigate the cases when some of the assumptions we impose in our theoretical results are not met. Last but not least, we present some open problems and conjectures in  \ref{sec:Simulations}. To keep the presentation as clear as possible, we present some technical results in the Appendix \ref{sec:Appendix}.

\section{Preliminaries}\label{sec:Preliminaries}

\subsection{Main definitions}
We now introduce this paper's key concepts: greedy sampling, voting scheme, voting power, fairness, and asymptotic fairness.

We start with defining greedy sampling. We consider  a probability distribution $P = (p_i)_{i \in \bbN}$  on $\bbN$ and an integer $k \in \bbN$. We sample with replacement until $k$ different nodes (or integers) are chosen.   The number of samplings needed to choose $k$ different nodes is given by
\begin{align}\label{eq:def_of_vkP}
	\begin{aligned}
		v_{k}:=v_k^{(P)} := 
		& \textnormal{ the number of samplings with replacement from  } \\
		& \textnormal{ distribution } P \textnormal{ until } k \textnormal{ different nodes are sampled}.
	\end{aligned}
\end{align}
The outcome of a sampling will be denoted by the multi-set
\begin{equation*}
		S :=\{ a_1, a_2, \ldots, a_{v_{k}} \};
\end{equation*}
here the $a_{i}$'s take values in $\bbN$.
Furthermore, for any $i \in \bbN$, let
	\begin{equation}\label{eq:def_of_AkP}
		A_k(i) := A_k^{(P)}(i) := \# \{j \in \{1, 2, \ldots, v_{k}\} : a_j = i\}
	\end{equation}
	be the number of occurrences of $i$  in the multi-set  $S=\{a_1, a_2, \ldots, a_{v_{k}}\}$. 

Every node $i$ is assigned a weight $m_{i}$. Together with a function $\FJADEF{f}{[0, \infty)}{[0, \infty)}$, that we call sampling weight function, the weights define a probability distribution $P = (p_i)_{i \in \bbN}$  on $\bbN$ by
\begin{equation*}
	p_i = \frac{f(m_i)}{\sum_{j= 1}^{\infty} f(m_j)}.
\end{equation*}


We consider a second weight function $\FJADEF{g}{[0, \infty)}{[0, \infty)}$, the averaging weight function, that weighs the samples opinions, see  Equation (\ref{eq:meanOpinion}). The couple $(f,g)$ is called a voting scheme. We first consider general voting schemes but focus later on the voting scheme $(f, g)$ with $f \equiv id$ and $g \equiv 1$. 

Let us denote by $S_{i}$ the multi-set of the sample for a given node $i$.
To define the voting powers of the nodes, we recall the definition of the mean opinion, Equation (\ref{eq:meanOpinion}), 
\begin{equation*}
\eta_{i} = \frac{\sum_{j \in S_{i}} g(m_{j}) s_{j}}{\sum_{j \in S_{i}} g(m_{j})}.
\end{equation*}
The multi-set $S_{i}$ is a random variable. Taking  expectation leads to 
\begin{equation*}
\bbE[\eta_{i}] =\bbE\left[  \sum_{j\in \bbN}  \frac{g(m_{j}) A_{k}(j) s_j}{\sum_{\ell\in \bbN} {g(m_{\ell}) A_{k}(\ell)}}\right]= \sum_{j\in \bbN} \bbE\left[  \frac{g(m_{j}) A_{k}(j) s_j}{\sum_{\ell\in \bbN} {g(m_{\ell}) A_{k}(\ell)}}\right].
\end{equation*}

Hence, the influence of the node $j$ on another node's mean opinion is measured by the corresponding coefficient in the above series.

\begin{DEF}[Voting power]
The \textnormal{voting power} of a node $i$ is defined as
	\begin{equation*}
		V_k(i) := V_k^{(P)}(m_{i}) :=  \bbE\left[  \frac{g(m_{i}) A_{k}(i)}{\sum_{\ell\in \bbN} {g(m_{\ell}) A_{k}(\ell)}}\right].
	\end{equation*}
If $g \equiv 1$, the voting power reduces to
\begin{equation*}
		V_k(i) = V_k^{(P)}(i) =  \bbE\UGL{\frac{A_k(i)}{v_{k}}}.
	\end{equation*} 
\end{DEF}

\begin{DEF}[$r$-splitting]\label{def:rsplitting}
Let $(m_i)_{i \in \bbN}$ be the weight distribution of the nodes and let $k \in \bbN$ be a positive integer. We fix some node $i$ and $r\in\bbN$. We say that   $m_{i_1^{(r)}}, \ldots, m_{i_r^{(r)}} > 0$ is an $r$-splitting of node $i$ if $m_i = \sum_{j = 1}^r m_{i_j^{(r)}}$. The probability distribution $P = (p_i)_{i \in \bbN}$, given in \eqref{eq:prob_from_mana}, changes to the probability distribution of the weights with $r$-splitting of node $i$ given by
	\begin{equation*}
		\widehat{P}_{r,i} := 
(\widehat{p}_1, \ldots, \widehat{p}_{i - 1}, \widehat{p}_{i_1^{(r)}}, \ldots, \widehat{p}_{i_r^{(r)}}, \widehat{p}_{i + 1}, \ldots)
	\end{equation*}
 on  $\{1, \ldots, i - 1, i_1^{(r)}, \ldots, i_r^{(r)}, i + 1, \ldots\}$, where
\begin{align*}
	\widehat{p}_j
	& = \frac{f(m_j)}{\sum_{u \in \bbN \setminus \{i\}} f(m_u) + \sum_{u = 1}^r f(m_{i_u^{(r)}})}, \qquad j \neq i, \\
	\widehat{p}_{i_j^{(r)}} & = \frac{f(m_{i_j^{(r)}})}{\sum_{u \in \bbN \setminus \{i\}} f(m_u) + \sum_{u = 1}^r f(m_{i_u^{(r)}})}, \qquad j \in \{1, 2, \ldots, r\}.
\end{align*}
\end{DEF}

\begin{DEF}[Fairness]\label{def:fairness}
We say that a voting scheme $(f, g)$ is
	\begin{enumerate}[(i)]
		\item \textnormal{robust to splitting into $r$ nodes} if for all nodes $i$ and all $r$-splittings $m_{i_1^{(r)}}, \ldots, m_{i_r^{(r)}}$ we have
\begin{equation}\label{eq:robustness_to_splitting}
			V_k^{(P)}(m_i) \ge \sum_{j = 1}^r V_k^{(\widehat{P}_{r,i})}(m_{i_j^{(r)}});
		\end{equation}
		\item \textnormal{robust to merging of $r$ nodes} if for all nodes $i$ and all $r$-splittings $m_{i_1^{(r)}}, \ldots, m_{i_r^{(r)}}$ we have	
		\begin{equation}\label{eq:robustness_to_merging}
			V_k^{(P)}(m_i) \le \sum_{j = 1}^r V_k^{(\widehat{P}_{r,i})}(m_{i_j^{(r)}}).
		\end{equation}
	\end{enumerate}
If Relation \eqref{eq:robustness_to_splitting} holds for every $r \in \bbN$, we say that the voting scheme $(f, g)$ is \textnormal{robust to splitting} and if Relation \eqref{eq:robustness_to_merging} holds for every $r \in \bbN$, we say that the voting scheme $(f, g)$ is \textnormal{robust to merging}. If a voting scheme $(f, g)$ is robust to splitting and robust to merging, that is, if for every  node $i$ and every $r \in \bbN$ and every $r$-splitting  $m_{i_1^{(r)}}, \ldots, m_{i_r^{(r)}} > 0$  it holds that
	\begin{equation*}
		V_k^{(P)}(m_i) = \sum_{j = 1}^r V_k^{(\widehat{P}_{r,i})}(m_{i_j^{(r)}}),
	\end{equation*}
we say that the voting scheme $(f, g)$ is \textnormal{fair}.
\end{DEF}

To generalize the above definitions to sequences of weights and to define asymptotic fairness, we first define sequence of $r$-splittings.
\begin{DEF}[Sequence of $r$-splittings]\label{def:seqrsplitting}
Let $k \in \bbN$ be a positive integer and let $\{(m_i^{(n)})_{i \in \bbN}\}_{n \in \bbN}$ be a sequence of weight distributions. Furthermore, for a fixed positive integer $r \in \bbN$ and  a fixed node $i$, we say that $m_{i_1^{(r)}}^{(n)}, \ldots, m_{i_r^{(r)}}^{(n)} > 0$ is a sequence of $r$-splittings of node $i$ if $m_i^{(n)} = \sum_{j = 1}^r m_{i_j^{(r)}}^{(n)}$. 
We define the sequence of probability distributions on the set $\{1, \ldots, i - 1, i_1^{(r)}, \ldots, i_r^{(r)}, i + 1, \ldots\}$, by
\begin{equation*}
		\widehat{P}_{r,i}^{(n)} := 
		(\widehat{p}_1^{(n)}, \ldots, \widehat{p}_{i - 1}^{(n)}, \widehat{p}_{i_1^{(r)}}^{(n)}, \ldots, \widehat{p}_{i_r^{(r)}}^{(n)}, \widehat{p}_{i + 1}^{(n)}, \ldots),
\end{equation*}		
with
\begin{align*}
		\widehat{p}_j^{(n)}
		& = \frac{f(m_j^{(n)})}{\sum_{u \in \bbN \setminus \{i\}} f(m_u^{(n)}) + \sum_{u = 1}^r f(m_{i_u^{(r)}}^{(n)})}, \qquad j \neq i, \\
		\widehat{p}_{i_j^{(r)}}^{(n)} & = \frac{f(m_{i_j^{(r)}}^{(n)})}{\sum_{u \in \bbN \setminus \{i\}} f(m_u^{(n)}) + \sum_{u = 1}^r f(m_{i_u^{(r)}}^{(n)})}, \qquad j \in \{1, 2, \ldots, r\}.
	\end{align*}
\end{DEF}

\begin{DEF}[Asymptotic fairness]\label{def:asymptotic_fairness}
	We say that a voting scheme $(f, g)$ is \textnormal{asymptotically fair for the sequence} $\{(m_i^{(n)})_{i \in \bbN}\}_{n \in \bbN}$ of weight distributions if for all $r$ and all nodes $i$,
	\begin{equation*}
		\APS{\sum_{j = 1}^r V_k^{(\widehat{P}_{r,i}^{(n)})}(m_{i_j^{(r)}}^{(n)}) - V_k^{(P^{(n)})}(m_i^{(n)})} \xrightarrow[n \rightarrow \infty]{} 0,
	\end{equation*}
for all sequences of $r$-splittings of node $i$. 
\end{DEF}

\begin{REM}
	The canonical class of examples of the sequence $\{(m_i^{(n)})_{i \in \bbN}\}_{n \in \bbN}$ of weight distributions is the one where $m_i^{(n)} = 0$ for all $i > n$. With these type of sequences of weight distributions, we can model the scenario where the number of nodes in the network grows to infinity.
\end{REM}

\subsection{Zipf's law}\label{subsec:Zipf_law}

We do not assume any particular weight distribution in our theoretical results. However, for examples and numerical simulation, it is essential to consider specific weight distributions. 

Probably the most appropriate modelings of  weight distributions rely on universality phenomena. The most famous example of this universality phenomenon is the central limit theorem. While the central limit theorem is suited to describe statistics where values are of the same order of magnitude, it is not appropriate to model more heterogeneous situations where the values might differ in several orders of magnitude. A Zipf law may describe heterogeneous weight distributions.  Zipf's law was first observed in quantitative linguistics, stating that any word's frequency is inversely proportional to its rank in the corresponding frequency table. Nowadays, many fields claim that specific data fits a Zipf law; e.g., city populations, internet traffic data, the formation of peer-to-peer communities, company sizes, and science citations. We refer to \cite{Li2002ZipfsLE} for a brief introduction and more references, and to \cite{Adamic2002ZipfsLA} for the appearance of Zipf's law in the internet and computer networks. We also refer to \cite{Tao} for a more mathematical introduction to this topic.

There is a ``rule of thumb'' for situations when a Zipf law may govern the asymptotic distribution of a data or statistic: variables 
\begin{itemize}
	\item[(1)] take values as positive numbers;
	\item[(2)] range over many different orders of magnitude;
	\item[(3)] arise from a complicated combination of largely independent factors; and
	\item[(4)] have not been artificially rounded, truncated, or otherwise constrained in size.
\end{itemize}
We consider a situation with $n$ elements or nodes. Zipf's law  predicts that the (normalized) frequency of the node of rank $k$  is given by \begin{equation}\label{eq:Zipf_law}
	y(k) := \frac{k^{-s}}{ \sum_{i = 1}^n i^{-s}},
\end{equation}
where $s\in [0,\infty)$ is the Zipf parameter. Since the value $y(k)$ in \eqref{eq:Zipf_law} only depends on two parameters, $s$ and $n$, this provides a convenient model to investigate the performance of a voting protocol in a wide range of network situations. For instance, nodes with equal weight can be modeled by choosing $s = 0$,  while more centralized networks can be described with parameters $s > 1$.

A convenient way to observe a Zipf law is by plotting the data on a log-log graph, with the axes being log(rank order) and log(value). The data conforms to a Zipf law to the extent that the plot is linear, and the value of $s$ may be estimated  using linear regression. We note that this visual inspection of the log-log plot of the ranked data is not a rigorous procedure. We refer to the literature on how to detect systematic modulation of the basic Zipf law and on how to fit more accurate models. In this work, we deal with distributions that are ``Zipf like'' without verifying certain test conditions.

For instance, Figure \ref{fig:Zipf_law} shows the distribution of IOTA for the top $10.000$ richest addresses with a fitted Zipf law.
\begin{figure}[ht]
	\centering
	\includegraphics[scale=0.4]{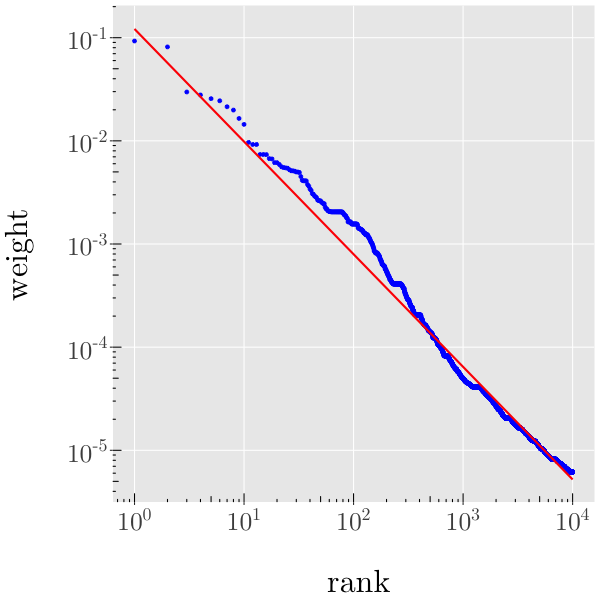}
	\caption{Relative distribution of top  $10.000$ IOTA addresses with a fitted Zipf distribution with $s = 1.1$, July 2020.}\label{fig:Zipf_law}
\end{figure}

Due to the universality phenomenon, the plausibility of hypotheses 1) - 4) above, and Figure \ref{fig:Zipf_law},  we assume the weight distribution to follow a Zipf law if we want to specify a weight distribution. To be more precise, we assume that for every $n \in \bbN$ and some parameter $s  > 0$
\begin{equation}\label{eq:seq_of_Zipf_law}
	p_j^{(n)} := 
	\begin{cases}
		\frac{1/j^s}{\sum_{i = 1}^n (1/i^s)}, \quad & j \le n, \\
		\hfil 0, \quad & j > n,
	\end{cases}
\end{equation}
where $P^{(n)} = (p_j^{(n)})_{j \in \bbN}$ is the weight distribution among the nodes in the network when the total number of nodes is $n$. Notice that, for a fixed $j$, the sequence $(p_j^{(n)})_{n \in \bbN}$ is decreasing in $n$. Furthermore, since $\sum_{i = 1}^{\infty} (1/i^s)$ diverges for $s \le 1$ the sequence $(p_j^{(n)})_{n \in \bbN}$ converges to $0$ in this case (when $n$ goes to infinity). On the other hand, if the parameter $s$ is strictly larger than $1$, the sequence $(p_j^{(n)})_{n \in \bbN}$ converges to a positive number (when $n$ goes to infinity).

\section{Greedy weighted sampling}\label{sec:Dist_of_sample_size}

We consider sampling with replacement until $k$ different  elements are chosen. The actual size of the sample is described by the random variable $v_{k}$. 

\begin{PROP}\label{prop:dist_of_v-k-P}
	Let $P = (p_i)_{i \in \bbN}$ be a probability distribution on $\bbN$, $k \in \bbN$ a positive integer and $v_{k}=v_k^{(P)}$ the random variable defined in \eqref{eq:def_of_vkP}. For every $v \in \{k, k + 1, k + 2, \ldots\}$ we have
	\begin{equation}\label{eq:dist_of_v-k-P}
		\bbP \OBL{v_k = v} = \sum_{i = 1}^{\infty} p_i \sum_{\substack{x_1 + \cdots + x_{k - 1} = v - 1 \\ x_1, \ldots, x_{k - 1} \ge 1}} \binom{v - 1}{x_1, \ldots, x_{k - 1}} \sum_{\substack{A \subset \bbN \setminus \{i\} \\ \aps{A} = k - 1}} (p_{a_1})^{x_1} \cdots (p_{a_{k - 1}})^{x_{k - 1}},
	\end{equation}
	where
\begin{equation}\label{eq:def_of_multinomial_coef}
		\binom{v - 1}{x_1, x_2, \ldots, x_{k - 1}} =
		\begin{cases}
			\frac{(v - 1)!}{x_1! x_2! \cdots x_{k - 1}!}, \quad & x_1 + x_2 + \ldots + x_{k - 1} = v - 1, \\
			\hfil 0, \quad & \textnormal{otherwise}.
		\end{cases}
	\end{equation}
\end{PROP}
\begin{proof}
We are sampling from the distribution $P$ until we sample $k$ different nodes. A first observation is that  the last node will be sampled only once. 
All the nodes that appear before the last one can be sampled more than once. We can construct such a sampling in the following way: first we choose a node $i \in \bbN$ that will be sampled the last, then we choose $k - 1$ different nodes $a_1, a_2, \ldots, a_{k - 1}$ from the set $\bbN \setminus \{i\}$ that will appear in the sequence before the last node and we choose positive integers $x_1, x_2, \ldots, x_{k - 1} \in \bbN$ that represent how many times each of the $k - 1$ nodes from the set $\{a_1, a_2, \ldots, a_{k - 1}\}$ will appear in the sampled sequence. Notice that $\sum_{i = 1}^{k - 1} x_i$ has to be equal to $v - 1$ because the total length of the sequence, including the last node $i$, has to be $v$. The last thing we need to choose is the permutation of the first $v - 1$ elements in the sequence which can be done in $\binom{v - 1}{x_1, x_2, \ldots, x_{k - 1}}$ ways. Summarizing, the probability of sampling a sequence where the last node is $i$ and first $k - 1$ nodes are $a_1, a_2, \ldots, a_{k - 1}$ and they appear $x_1, x_2, \ldots, x_{k - 1}$ times is
	\begin{equation*}
		p_i \binom{v}{x_1, x_2, \ldots, x_{k - 1}} (p_{a_1})^{x_1} (p_{a_2})^{x_2} \cdots (p_{a_{k - 1}})^{x_{k - 1}}.
	\end{equation*}
	Now, we need to sum this up with respect to all the possible values of the element $i$, all the possible sequences of $k - 1$ positive integers $x_1, x_2, \ldots, x_{k - 1}$ that sum up to $v - 1$ (i.e., all the partitions of the integer $v - 1$ into $k - 1$ parts) and all the subsets of $\bbN \setminus \{i\}$ of cardinality $k - 1$. This gives us exactly the expression from Equation \eqref{eq:dist_of_v-k-P}.
\end{proof}

\begin{REM}
The random variable $v_k^{(P)}$ was  studied in \cite{RaKh:58} in the case where the population is finite and elements have equal weight.  Therefore, Formula \eqref{eq:dist_of_v-k-P} is a generalization of \cite[Formula (16)]{RaKh:58}.
	
Another  random variable studied in \cite{RaKh:58} is the number of  different elements  in a sample with replacement of a fixed  size. To be precise, let $k \in \bbN$ be a positive integer and $P = (p_i)_{i \in \bbN}$ be a probability distribution on $\bbN$. Denote with
	\begin{align*}
		u_k^{(P)} = 
		& \textnormal{ the number of different nodes sampled in } k \\
		& \textnormal{ samplings with replacement from distribution } P.
	\end{align*}
The authors in \cite{RaKh:58} calculated the distribution of the random variable $u_k^{(P)}$, but again under the assumptions that the set from which the elements are sampled is finite and that all the elements are sampled with the same probability. Using analogous reasoning as in the proof of Proposition \ref{prop:dist_of_v-k-P}, for $u \in \{1, 2, \ldots, k\}$, we  get 
	\begin{equation*}
		\bbP(u_k^{(P)} = u) = \sum_{\substack{x_1 + \cdots + x_u = k \\ x_1, \ldots, x_u \ge 1}} \binom{k}{x_1, \ldots, x_u} \sum_{\substack{A \subset \bbN \\ \aps{A} = u}} (p_{a_1})^{x_1} \cdots (p_{a_u})^{x_u}.
	\end{equation*}
	This formula generalizes \cite[Formula (8)]{RaKh:58}.
\end{REM}

Using Proposition \ref{prop:dist_of_v-k-P}, we now find the distribution of the random vector $(A_k^{}(i), v_k)$ for all $i \in \bbN$. 

\begin{PROP}\label{prop:dist_of_Akp-Vkp}
	Let $P = (p_i)_{i \in \bbN}$ be a probability distribution on $\bbN$, $k \in \bbN$ a positive integer, $v_{k}=v_k^{(P)}$ the random variable defined in \eqref{eq:def_of_vkP} and $A_k^{}(i)=A_k^{(P)}(i)$ the random variable defined in \eqref{eq:def_of_AkP}. The support of the random vector $(A_k^{}(i), v_k^{})$ is
	\begin{equation*}
		\{(0, v) : v \ge k\} \cup \{(\ell, v) : 1 \le \ell \le  v - k + 1\}.
	\end{equation*}
For every node $i \in \bbN$ and every $(\ell, v)$ in the support of $(A_k^{}(i), v_k^{})$  we have
	\begin{equation}\label{eq:dist_of_Akp-Vkp}
		\bbP(A_k^{}(i) = \ell, v_k^{} = v) = 
		\begin{cases}
			\sum\limits_{\substack{j = 1 \\ j \neq i}}^{\infty} p_j \sum\limits_{\substack{x_1 + \cdots + x_{k - 1} = v - 1\\ x_1, \ldots, x_{k - 1} \ge 1}} \binom{v - 1}{x_1, \ldots, x_{k - 1}} \sum\limits_{\substack{A \subset \bbN \setminus \{i, j\} \\ \aps{A} = k - 1}} \prod\limits_{r = 1}^{k - 1} (p_{a_r})^{x_r}, &  \hspace{-0.2cm} \ell = 0, \\
			\sum\limits_{\substack{j = 1 \\ j \neq i}}^{\infty} p_j \Big(\sum\limits_{\substack{x_1 + \cdots + x_{k - 2}  = v - 2 \\ x_1, \ldots, x_{k - 2} \ge 1}} \binom{v - 1}{x_1, \ldots, x_{k - 2}, 1} p_i \sum\limits_{\substack{A \subset \bbN \setminus\{i, j\} \\ \aps{A} = k - 2}} \prod\limits_{r = 1}^{k - 2}(p_{a_r})^{x_r}\Big) \,+ \\
			\quad p_i \sum\limits_{\substack{x_1 + \cdots + x_{k - 1} = v - 1 \\ x_1, \ldots, x_{k - 1} \ge 1}} \binom{v - 1}{x_1, \ldots, x_{k - 1}} \sum\limits_{\substack{A \subset \bbN \setminus\{i\} \\ \aps{A} = k - 1}} \prod\limits_{r = 1}^{k - 1} (p_{a_r})^{x_r}, & \hspace{-0.2cm}  \ell = 1, \\
			\sum\limits_{\substack{j = 1 \\ j \neq i}}^{\infty} p_j \sum\limits_{\substack{x_1 + \cdots + x_{k - 2}  = v - \ell- 1 \\ x_1, \ldots, x_{k - 2} \ge 1}} \binom{v - 1}{x_1, \ldots, x_{k - 2}, \ell} (p_i)^\ell \sum\limits_{\substack{A \subset \bbN \setminus\{i, j\} \\ \aps{A} = k - 2}} \prod\limits_{r = 1}^{k - 2}(p_{a_r})^{x_r}, & \hspace{-0.2cm} \ell \ge 2.
		\end{cases}
	\end{equation}
\end{PROP}
\begin{proof}
Notice first that $(0, v)$ is in the support of $(A_k^{}(i), v_k^{})$. Now,  if $\ell \ge 1$ then for all $v < \ell + k - 1$ we have that $\bbP(A_k^{}(i) = l, v_k^{} = v) = 0$ since we need at least $\ell + k - 1$ samplings to sample  $\ell$ times node $i$ and the other $k - 1$ different nodes at least once. 

Let us consider separately different values of non-negative integer $\ell \in \bbN \cup \{0\}$.
	
	$\ell = 0$: This case is an immediate consequence of Proposition \ref{prop:dist_of_v-k-P}. We just need to restrict the set of all nodes that can be sampled to $\bbN \setminus \{i\}$.
	
	$\ell= 1$: Here we need to distinguish two disjoint scenarios. First one is when the node $i$ is not sampled as the last node (i.e., node $i$ is not the $k$-th different node that has been sampled). This means that the node $i$ was sampled in the first $v - 1$ samplings. Hence, we first choose node $j \in \bbN \setminus \{i\}$ that will be sampled the last. Then we choose $k - 2$ different nodes $a_1, a_2, \ldots, a_{k - 2}$ from the set $\bbN \setminus \{i, j\}$ that will appear (together with the node $i$) in the sampled sequence before the last node and we choose positive integers $x_1, x_2, \ldots, x_{k - 2} \in \bbN$ that represent how many times each of the $k - 2$ nodes from the set $\{a_1, a_2, \ldots, a_{k - 2}\}$ will appear in the sampled sequence. Notice that $\sum_{i = 1}^{k - 2} x_i$ has to be equal to $v - 2$ because the total length of the sequence, including one appearance of node $j$ (on the last place) and one appearance of node $i$ (somewhere in the first $v - 1$ samplings), has to be $v$. The last thing we need to choose is the permutation of the first $v - 1$ nodes in the sequence which can be done in $\binom{v - 1}{x_1, \ldots, x_{k - 2}, 1}$ ways (taking into consideration that node $i$ appears only once). Summarizing, the probability of sampling a sequence where the last node is $j \neq i$, node $i$ appears exactly once in the first $v - 1$ sampled nodes and the rest $k - 2$ nodes that appear together with the node $i$ before the last node $j$ are $a_1, a_2, \ldots, a_{k - 2}$ and they appear $x_1, x_2, \ldots, x_{k - 2}$ times is
	\begin{equation*}
		p_j \binom{v - 1}{x_1, \ldots, x_{k - 2}, 1} p_i \prod_{r = 1}^{k - 2} (p_{a_r})^{x_r}.
	\end{equation*}
	As in Proposition \ref{prop:dist_of_v-k-P}, we now sum this up with respect to all the possible values of the node $j$, all the possible sequences of $k - 2$ positive integers $x_1, x_2, \ldots, x_{k - 2}$ that sum up to $v - 2$ and all the subsets of $\bbN \setminus \{i, j\}$ of cardinality $k - 2$. This way we obtain the first term in the expression for $\bbP(A_k^{}(i) = 1, v_k^{} = v)$. The second scenario is the one where the node $i$ is sampled the last. Here the situation is much simpler. The last node is fixed to be $i \in \bbN$ and then we choose $k - 1$ nodes that appear before, and the number of times they appear analogously as in Proposition \ref{prop:dist_of_v-k-P}. We immediately get the second term in the expression for $\bbP(A_k^{}(i) = 1, v_k^{} = v)$.
	
	$\ell \ge 2$: Notice that in this case we don't have two different scenarios because it is impossible that the node $i$ was sampled the last. As we explained in Proposition \ref{prop:dist_of_v-k-P}, the last node can be sampled only once since we terminate sampling when we reach $k$ different nodes. Now we reason analogously as in the first scenario of the case $\ell = 1$. The only difference is that here node $i$ appears $\ell$ times (in the first $v - 1$ samplings) so the integers $x_1, x_2, \ldots, x_{k - 2}$ have to satisfy $\sum_{i = 1}^{k - 2} x_i = v - \ell - 1$. Together with $\ell$ appearances of the node $i$ and one appearance of the last node, this gives $v$ sampled nodes in total.
\end{proof}

Let $b = (b_i)_{i \in \bbN} \subset \bbR$ be a sequence of real numbers. Denote with $\norm{b}_{\infty} = \sup_{i \in \bbN} \aps{b_i}$ the supremum norm of the sequence $b$. The next result shows that if the probabilities of sampling each of the nodes converge uniformly to zero then the number of samplings needed to sample $k$ different elements converges to $k$. 

\begin{LM}\label{lm:v_k^n->k}
	Let $(P^{(n)})_{n \in \bbN}$, $P^{(n)} = (p_i^{(n)})_{i \in \bbN}$, be a sequence of probability distributions on $\bbN$ and let $(k_n)_{n \in \bbN}$ be a sequence of positive integers such that  $k_n^2 \norm[1]{P^{(n)}}_{\infty} \xrightarrow[n \rightarrow \infty]{} 0$. Then, $v_{k_n}^{(P^{(n)})} / k_n \xrightarrow[n \rightarrow \infty]{\bbP} 1$. In particular, if   for some fixed positive $k \in \bbN$ we have $k_n = k$ for all $n \in \bbN$, and $\norm[1]{P^{(n)}}_{\infty} \xrightarrow[n \rightarrow \infty]{} 0$, then we have that $v_k^{(P^{(n)})} \xrightarrow[n \rightarrow \infty]{\bbP} k$.
\end{LM}
\begin{proof}
	For simplicity, we denote $v_{k_{n}}^{(n)} := v_{k_{n}}^{(P^{(n)})}$. Since $v_{k_n}^{(n)}$ is larger than or equal to $k_n$, it is sufficient to show that $\bbP(v_{k_n}^{(n)} / k_n > 1) \xrightarrow[n \rightarrow \infty]{} 0$. Denote by $X_i^{(n)}$ the random variable representing the node sampled in the $i$-th sampling. Since the event $\{v_{k_n}^{(n)} / k_n  > 1\}$ happens if and only if some of the nodes sampled in the first $k_n$ samplings appear more that once, we have
	\begin{align*}
		\bbP(v_{k_n}^{(n)} / k_n > 1) &  \\
		&\hspace{-1cm} = \bbP(X_1^{(n)} = X_2^{(n)}) + \bbP(X_1^{(n)} \neq X_2^{(n)}, X_3^{(n)} \in \{X_1^{(n)}, X_2^{(n)}\}) + \cdots \\
		&  \hspace{-1cm}\cdots + \bbP(X_1^{(n)} \neq X_2^{(n)}, X_3^{(n)} \notin \{X_1^{(n)}, X_2^{(n)}\}, \ldots, X_{k_n}^{(n)} \in \{X_1^{(n)}, X_2^{(n)}, \ldots, X_{k_n - 1}^{(n)}\}) \\
		& \hspace{-1cm}= \sum_{i_1 \in \bbN} (p_{i_1}^{(n)})^2 + \sum_{i_1 \in \bbN} \sum_{\substack{i_2 \in \bbN \\ i_2 \neq i_1}} p_{i_1}^{(n)} p_{i_2}^{(n)} (p_{i_1}^{(n)} + p_{i_2}^{(n)})\\
		& \hspace{-1cm}\qquad + \ldots + \sum_{i_1 \in \bbN} \sum_{\substack{i_2 \in \bbN \\ i_2 \neq i_1}} \cdots \sum_{\substack{i_{k_n - 1} \in \bbN \\ i_{k_n - 1} \neq i_1 \\ \cdots \\ i_{k_n - 1} \neq i_{k_n - 2}}} p_{i_1}^{(n)} p_{i_2}^{(n)} \cdots p_{i_{k_n - 1}}^{(n)} \sum_{j = 1}^{k_n - 1} p_{i_j}^{(n)} \\
		& \hspace{-1cm}\le \norm[1]{P^{(n)}}_{\infty} \sum_{i_1 \in \bbN} p_{i_1}^{(n)} + 2\norm[1]{P^{(n)}}_{\infty} \sum_{i_1 \in \bbN} \sum_{i_2 \in \bbN} p_{i_1}^{(n)} p_{i_2}^{(n)}\\
		& \hspace{-1cm}\qquad + \ldots + (k_n - 1)\norm[1]{P^{(n)}}_{\infty} \sum_{i_1 \in \bbN} \sum_{i_2 \in \bbN} \cdots \sum_{i_{k_n - 1} \in \bbN} p_{i_1}^{(n)} p_{i_2}^{(n)} \cdots p_{i_{k_n - 1}}^{(n)} \\
		& \hspace{-1cm}= \norm[1]{P^{(n)}}_{\infty} (1 + 2 + \cdots + k_n - 1) \le k_n^2 \norm[1]{P^{(n)}}_{\infty}.
	\end{align*}
By the assumption, the last term converges to zero when $n$ goes to infinity, which is exactly what we wanted to prove.
\end{proof}
\begin{REM}
	Let us investigate what happens when the sequence $(P^{(n)})_{n \in \bbN}$ is defined by a Zipf law (see \eqref{eq:seq_of_Zipf_law})  with parameter $s > 0$. Since each of the sequences $(p_i^{(n)})_{i \in \bbN}$ is decreasing in $i$ we have
	\begin{equation*}
		\norm[1]{P^{(n)}}_{\infty} = p_1^{(n)} = \frac{1}{\sum_{i = 1}^n \frac{1}{i^s}}.
	\end{equation*}
	Notice that for all $s \le 1$ we have $\norm[1]{P^{(n)}}_{\infty} \xrightarrow[n \rightarrow \infty]{} 0$ because the series $\sum_{i = 1}^{\infty} \frac{1}{i^s}$ diverges for those values of the parameter $s$. Hence, for a fixed integer $k \in \bbN$, we have that $v_k^{(P^{(n)})} \xrightarrow[n \rightarrow \infty]{\bbP} k$ whenever $s \le 1$. Another important example is when sequence $(k_n)_{n \in \bbN}$ is given by
	\begin{equation*}
		k_n = \floor{\log(n)},
	\end{equation*}
	where for $x \in \bbR$, $\floor{x}$ is the largest integer less than or equal to $x$. Using
	\begin{equation*}
		\sum_{i = 1}^{n} \frac{1}{i^s} \sim
		\begin{cases}
			n^{1 - s}, \quad & s < 1, \\
			\hfil \log(n), \quad & s = 1,
		\end{cases}
	\end{equation*}
	we get, for $s < 1$, $k_n^2 \norm[1]{P^{(n)}}_{\infty} \xrightarrow[n \rightarrow \infty]{} 0$ so we can apply Lemma \ref{lm:v_k^n->k} for this particular choice of sequences $(k_n)_{n \in \bbN}$ and $(P^{(n)})_{n \in \bbN}$. 
\end{REM}

In Lemma \ref{lm:v_k^n->k} we dealt with the behavior of the sequence of random variables $(v_{k_{n}}^{P^{(n)}})_{n \in \bbN}$ if the sequence $(P^{(n)})_{n \in \bbN}$ 
satisfies $\,\,\norm[1]{P^{(n)}}_{\infty} \xrightarrow[n \rightarrow \infty]{} 0$. Next, we study the case when the sequence $(P^{(n)})_{n \in \bbN}$ converges in the supremum norm to another probability distribution $P^{(\infty)}$ on $\bbN$. As before, for $b = (b_i)_{i \in \bbN} \subset \bbR$, we use the notation $\norm{b}_{\infty} = \sup_{i \in \bbN} \aps{b_i}$ and we write $\norm{b}_1 = \sum_{i = 1}^{\infty} \aps{b_i}$.

\begin{PROP}\label{prop:cvg_in_dist_of_(A,v)}
	Let $(P^{(n)})_{n \in \bbN}$, $P^{(n)} = (p_i^{(n)})_{i \in \bbN}$, be a sequence of probability distributions on $\bbN$ and let $P^{(\infty)} = (p_i^{(\infty)})_{i \in \bbN}$ be a probability distribution on $\bbN$. If $$\norm[1]{P^{(n)} - P^{(\infty)}}_{\infty} = \sup_{i \in \bbN} \aps{p_i^{(n)} - p_i^{(\infty)}} \xrightarrow[n \rightarrow \infty]{} 0$$ then, for all fixed $k \in \bbN$,
	\begin{equation*}
		(A_k^{(P^{(n)})}(i), v_k^{(P^{(n)})}) \xrightarrow[n \rightarrow \infty]{(d)} (A_k^{(P^{(\infty)})}(i), v_k^{(P^{(\infty)})}),
	\end{equation*}
	where $\xrightarrow[]{(d)}$ denotes convergence in distribution.
\end{PROP}
\begin{proof}
	For simplicity, we denote $v^{(n)}_{k} := v_k^{(P^{(n)})}$, $v^{(\infty)}_{k} := v_k^{(P^{(\infty)})}$, $A_{k}^{(n)}(i) := A_k^{(P^{(n)})}(i)$ and $A_k^{(\infty)}(i) := A_k^{(P^{(\infty)})}(i)$. Since we consider discrete random variables, the statement
	\begin{equation*}
		(A_k^{(n)}(i), v^{(n)}_{k}) \xrightarrow[n \rightarrow \infty]{(d)} (A_k^{(\infty)}(i), v^{(\infty)}_{k})
	\end{equation*}
	is equivalent to
	\begin{equation*}
		\bbP(A_k^{(n)}(i) = \ell, v^{(n)}_{k} = v) \xrightarrow[n \rightarrow \infty]{} \bbP(A_k^{(\infty)}(i) = \ell, v^{(\infty)}_{k} = v)
	\end{equation*}
	for all $\ell \in \bbN \cup \{0\}$ and all $v \in \bbN$. As in the proof of Proposition \ref{prop:dist_of_Akp-Vkp}, we consider separately different values of the non-negative integer $\ell \in \bbN \cup \{0\}$.
	
	$\ell= 0$: Using Proposition \ref{prop:dist_of_Akp-Vkp}, we have
	\begin{align*}
		\vert \bbP(A_k^{(n)}(i)
		& = 0, v^{(n)}_{k} = v) - \bbP(A_k^{(\infty)}(i) = 0, v^{(\infty)}_{k} = v) \vert \\
		& = \Bigg\vert \sum_{\substack{j = 1 \\ j \neq i}}^{\infty} p_j^{(n)} \sum_{\substack{x_1 + \cdots + x_{k - 1} = v - 1 \\ x_1, \ldots, x_{k - 1} \ge 1}} \binom{v - 1}{x_1, \ldots, x_{k - 1}} \sum_{\substack{A \subset \bbN \setminus \{i, j\} \\ \aps{A} = k - 1}} (p_{a_1}^{(n)})^{x_1} \cdots (p_{a_{k - 1}}^{(n)})^{x_{k - 1}} \\
		& \qquad - \sum_{\substack{j = 1 \\ j \neq i}}^{\infty} p_j^{(\infty)} \sum_{\substack{x_1 + \cdots + x_{k - 1} = v - 1 \\ x_1, \ldots, x_{k - 1} \ge 1}} \binom{v - 1}{x_1, \ldots, x_{k - 1}} \sum_{\substack{A \subset \bbN \setminus \{i, j\} \\ \aps{A} = k - 1}} (p_{a_1}^{(\infty)})^{x_1} \cdots (p_{a_{k - 1}}^{(\infty)})^{x_{k - 1}} \Bigg\vert \\
		& \le \sum_{\substack{x_1 + \cdots + x_{k - 1} = v - 1 \\ x_1, \ldots, x_{k - 1} \ge 1}} \binom{v - 1}{x_1, \ldots, x_{k - 1}} \Bigg\vert \Big( \sum_{\substack{j = 1 \\ j \neq i}}^{\infty} p_j^{(n)} \sum_{\substack{A \subset \bbN \setminus \{i, j\} \\ \aps{A} = k - 1}} (p_{a_1}^{(n)})^{x_1} \cdots (p_{a_{k - 1}}^{(n)})^{x_{k - 1}} \Big) \\
		& \qquad - \Big( \sum_{\substack{j = 1 \\ j \neq i}}^{\infty} p_j^{(\infty)} \sum_{\substack{A \subset \bbN \setminus \{i, j\} \\ \aps{A} = k - 1}} (p_{a_1}^{(\infty)})^{x_1} \cdots (p_{a_{k - 1}}^{(\infty)})^{x_{k - 1}} \Big) \Bigg\vert \\
		& = \sum_{\substack{x_1 + \cdots + x_{k - 1} = v - 1 \\ x_1, \ldots, x_{k - 1} \ge 1}} \binom{v - 1}{x_1, \ldots, x_{k - 1}} \aps{I^{(n)}(x_1, \ldots x_{k - 1}) - I^{(\infty)}(x_1, \ldots x_{k - 1})},
	\end{align*}
with 	
\begin{align*}
I^{(n)}(x_1, \ldots x_{k - 1}) & :=  \sum_{\substack{j = 1 \\ j \neq i}}^{\infty} p_j^{(n)} \sum_{\substack{A \subset \bbN \setminus \{i, j\} \\ \aps{A} = k - 1}} (p_{a_1}^{(n)})^{x_1} \cdots (p_{a_{k - 1}}^{(n)})^{x_{k - 1}}, \\
I^{(\infty)}(x_1, \ldots x_{k - 1}) &:= \sum_{\substack{j = 1 \\ j \neq i}}^{\infty} p_j^{(\infty)} \sum_{\substack{A \subset \bbN \setminus \{i, j\} \\ \aps{A} = k - 1}} (p_{a_1}^{(\infty)})^{x_1} \cdots (p_{a_{k - 1}}^{(\infty)})^{x_{k - 1}}.
\end{align*}
	
It remains to prove  that $\aps{I^{(n)}(x_1, \ldots x_{k - 1}) - I^{(\infty)}(x_1, \ldots x_{k - 1})} \xrightarrow[n \rightarrow \infty]{} 0$, uniformly for all possible values of positive integers $x_1, x_2, \ldots x_{k - 1}$. This is sufficient since  the number of partitions of integer $v - 1$ into $k - 1$ parts is finite and independent of $n$. Notice that for every $i, j \in \bbN$
\begin{equation}\label{eq:cvg_of_sum_over_A}
		\sum_{\substack{A \subset \bbN \setminus \{i, j\} \\ \aps{A} = k - 1}} (p_{a_1}^{(\infty)})^{x_1} \cdots (p_{a_{k - 1}}^{(\infty)})^{x_{k - 1}} \le \sum_{\substack{A \subset \bbN \\ \aps{A} = k - 1}} p_{a_1}^{(\infty)} \cdots p_{a_{k - 1}}^{(\infty)} \le \sum_{a_1 = 1}^{\infty} p_{a_1}^{(\infty)} \cdots \sum_{a_{k - 1} = 1}^{\infty} p_{a_{k - 1}}^{(\infty)} = 1.
	\end{equation}
	Clearly, the same is true when, instead of distribution $P^{(\infty)}$, we consider the  distribution $P^{(n)}$. Due to convergence of these series, we can rewrite 
	\begin{align*}
		I^{(n)}(x_1, \ldots, x_{k - 1})
			& - I^{(\infty)}(x_1, \ldots, x_{k - 1}) = \Big( \sum_{\substack{j = 1 \\ j \neq i}}^{\infty} (p_j^{(n)} - p_j^{(\infty)}) \sum_{\substack{A \subset \bbN \setminus \{i, j\} \\ \aps{A} = k - 1}} (p_{a_1}^{(n)})^{x_1} \cdots (p_{a_{k - 1}}^{(n)})^{x_{k - 1}} \Big) \\
		&  + \Big( \sum_{\substack{j = 1 \\ j \neq i}}^{\infty} p_j^{(\infty)} \sum_{\substack{A \subset \bbN \setminus \{i, j\} \\ \aps{A} = k - 1}} \OBL{(p_{a_1}^{(n)})^{x_1} \cdots (p_{a_{k - 1}}^{(n)})^{x_{k - 1}} - (p_{a_1}^{(\infty)})^{x_1} \cdots (p_{a_{k - 1}}^{(\infty)})^{x_{k - 1}}} \Big) \\
		& =: S_1^{(n)}(x_1, \ldots, x_{k - 1}) + S_2^{(n)}(x_1, \ldots, x_{k - 1}).
	\end{align*}
	For simplicity, we write $S_1^{(n)} = S_1^{(n)}(x_1, \ldots, x_{k - 1})$ and $S_2^{(n)} = S_2^{(n)}(x_1, \ldots, x_{k - 1})$. It remains to prove that $S_1^{(n)}$ and $S_2^{(n)}$ converge to $0$ when $n$ goes to infinity. Using Inequality \eqref{eq:cvg_of_sum_over_A} and Proposition \ref{prop:equiv_of_cvg_in_l-1_and_l-infty} we have that
	\begin{equation*}
		\aps{S_1^{(n)}} \le \sum_{j = 1}^{\infty} \aps{p_j^{(n)} - p_j^{(\infty)}} \cdot 1 \xrightarrow[n \rightarrow \infty]{} 0.
	\end{equation*}
To treat  the term $S_2^{(n)}$ we use Lemma \ref{lm:telescope_equality}, in the second line, and Proposition \ref{prop:equiv_of_cvg_in_l-1_and_l-infty}, in the last line, to obtain
	\begin{align*}
		\aps{S_2^{(n)}}
		& \le \sum_{\substack{j = 1 \\ j \neq i}}^{\infty} p_j^{(\infty)} \sum_{\substack{A \subset \bbN \setminus \{i, j\} \\ \aps{A} = k - 1}} \APS{(p_{a_1}^{(n)})^{x_1} \cdots (p_{a_{k - 1}}^{(n)})^{x_{k - 1}} - (p_{a_1}^{(\infty)})^{x_1} \cdots (p_{a_{k - 1}}^{(\infty)})^{x_{k - 1}}} \\
		& \le \sum_{j = 1}^{\infty} p_j^{(\infty)} \hspace{-0.3cm}\sum_{\substack{A \subset \bbN \\ \aps{A} = k - 1}} \left|\sum_{r = 1}^{k - 1} (p_{a_1}^{(n)})^{x_1} \cdots (p_{a_{r - 1}}^{(n)})^{x_{r - 1}} \OBL{(p_{a_r}^{(n)})^{x_r} -  (p_{a_r}^{(\infty)})^{x_r}}\cdot \right. \\
		& \hspace{8cm} \cdot \left.(p_{a_{r + 1}}^{(\infty)})^{x_{r + 1}} \cdots (p_{a_{k - 1}}^{(\infty)})^{x_{k - 1}} \right|\\
		& \le \sum_{r = 1}^{k - 1} \sum_{a_1 = 1}^{\infty} \cdots \sum_{a_{k - 1} = 1}^{\infty} (p_{a_1}^{(n)})^{x_1} \cdots (p_{a_{r - 1}}^{(n)})^{x_{r - 1}} \APS{(p_{a_r}^{(n)})^{x_r} - (p_{a_r}^{(\infty)})^{x_r}}\cdot \\
		& \hspace{8cm} \cdot (p_{a_{r + 1}}^{(\infty)})^{x_{r + 1}} \cdots (p_{a_{k - 1}}^{(\infty)})^{x_{k - 1}} \\
		& \le \sum_{r = 1}^{k - 1} \sum_{a_r = 1}^{\infty} \APS{(p_{a_r}^{(n)})^{x_r} - (p_{a_r}^{(\infty)})^{x_r}} \\
		& = \sum_{r = 1}^{k - 1} \sum_{a_r = 1}^{\infty} \APS{p_{a_r}^{(n)} -  p_{a_r}^{(\infty)}} \APS{(p_{a_r}^{(n)})^{x_r - 1} + (p_{a_r}^{(n)})^{x_r - 2} (p_{a_r}^{(\infty)}) + \cdots + (p_{a_r}^{(\infty)})^{x_r - 1}} \\
		& \le (k - 1)x_r \sum_{j = 1}^{\infty} \APS{p_j^{(n)} - p_j^{(\infty)}} \le vk \sum_{j = 1}^{\infty} \APS{p_j^{(n)} - p_j^{(\infty)}} \xrightarrow[n \rightarrow \infty]{} 0.
	\end{align*} 
	$\ell \ge 2$: Again using Proposition \ref{prop:dist_of_Akp-Vkp}, we have
	\begin{align*}
		\vert& \bbP(A_{k}^{(n)}(i)
		 = \ell,  v^{(n)}_{k} = v) - \bbP(A_k^{(\infty)}(i) = \ell, v^{(\infty)}_{k} = v) \vert \\
		& = \Bigg\vert \sum_{\substack{j = 1 \\ j \neq i}}^{\infty} p_j^{(n)} \sum_{\substack{x_1 + \cdots + x_{k - 2} = v - \ell - 1 \\ x_1, \ldots, x_{k - 2} \ge 1}} \binom{v - 1}{x_1, \ldots, x_{k - 2}, \ell} (p_i^{(n)})^\ell \sum_{\substack{A \subset \bbN \setminus \{i, j\} \\ \aps{A} = k - 2}} (p_{a_1}^{(n)})^{x_1} \cdots (p_{a_{k - 2}}^{(n)})^{x_{k - 2}} \\
		& \qquad - \sum_{\substack{j = 1 \\ j \neq i}}^{\infty} p_j^{(\infty)} \sum_{\substack{x_1 + \cdots + x_{k - 2} = v - \ell - 1 \\ x_1, \ldots, x_{k - 2} \ge 1}} \binom{v - 1}{x_1, \ldots, x_{k - 2}, \ell} (p_i^{(\infty)})^\ell \sum_{\substack{A \subset \bbN \setminus \{i, j\} \\ \aps{A} = k - 2}} (p_{a_1}^{(\infty)})^{x_1} \cdots (p_{a_{k - 2}}^{(\infty)})^{x_{k - 2}} \Bigg\vert \\
		& \le  \hspace{-0.5cm}\sum_{\substack{x_1 + \cdots + x_{k - 2} = v - \ell - 1 \\ x_1, \ldots, x_{k - 2} \ge 1}} \hspace{-0.2cm}\binom{v - 1}{x_1, \ldots, x_{k - 2}, \ell} \Bigg\vert  \hspace{-0.1cm}\sum_{\substack{j = 1 \\ j \neq i}}^{\infty} (p_j^{(n)} (p_i^{(n)})^\ell - p_j^{(\infty)} (p_i^{(\infty)})^\ell) \hspace{-0.32cm}\sum_{\substack{A \subset \bbN \setminus \{i, j\} \\ \aps{A} = k - 2}} \hspace{-0.28cm}(p_{a_1}^{(n)})^{x_1} \hspace{-0.1cm} \cdots (p_{a_{k - 2}}^{(n)})^{x_{k - 2}} \\
		& \quad + \sum_{\substack{j = 1 \\ j \neq i}}^{\infty} p_j^{(\infty)} (p_i^{(\infty)})^\ell \sum_{\substack{A \subset \bbN \setminus \{i, j\} \\ \aps{A} = k - 2}} \OBL{(p_{a_1}^{(n)})^{x_1} \cdots (p_{a_{k - 2}}^{(n)})^{x_{k - 2}} - (p_{a_1}^{(\infty)})^{x_1}\cdots (p_{a_{k - 2}}^{(\infty)})^{x_{k - 2}}} \Bigg\vert.
	\end{align*}
To show  that the above expression converges to zero as $n$ tends to infinity it remains to verify that
	\begin{equation*}
		\sum_{\substack{j = 1 \\ j \neq i}}^{\infty} \aps{p_j^{(n)} (p_i^{(n)})^\ell - p_j^{(\infty)} (p_i^{(\infty)})^\ell} \xrightarrow[n \rightarrow \infty]{} 0.
	\end{equation*}
To obtain this, we can use the same arguments as in the previous case. Again, introducing a middle term leads to
	\begin{align*}
		\sum_{\substack{j = 1 \\ j \neq i}}^{\infty} \vert p_j^{(n)} (p_i^{(n)})^\ell
		& - p_j^{(\infty)} (p_i^{(\infty)})^\ell \vert = \sum_{\substack{j = 1 \\ j \neq i}}^{\infty} \aps{p_j^{(n)} (p_i^{(n)})^\ell - p_j^{(n)} (p_i^{(\infty)})^\ell + p_j^{(n)} (p_i^{(\infty)})^\ell - p_j^{(\infty)} (p_i^{(\infty)})^\ell} \\
		& \le \aps{(p_i^{(n)})^\ell - (p_i^{(\infty)})^\ell} \sum_{j = 1}^{\infty} p_j^{(n)} + (p_i^{(\infty)})^\ell \sum_{j = 1}^{\infty} \aps{p_j^{(n)} - p_j^{(\infty)}} \\
		& \le \aps{p_i^{(n)} - p_i^{(\infty)}} \aps{(p_i^{(n)})^{\ell - 1} + (p_i^{(n)})^{\ell - 2} p_i^{(\infty)} + \cdots + (p_i^{(\infty)})^{\ell - 1}} + \sum_{j = 1}^{\infty} \aps{p_j^{(n)} - p_j^{(\infty)}} \\
		& \le \ell \cdot \norm[1]{P^{(n)} - P^{(\infty)}}_{\infty} + \norm[1]{P^{(n)} - P^{(\infty)}}_1.
	\end{align*}
	Applying again Proposition \ref{prop:equiv_of_cvg_in_l-1_and_l-infty} we get the desired result.
	
	$\ell= 1$: The above argument stays the same for $\ell= 1$. Hence, the difference of the first terms in the expressions for $\bbP(A_k^{(n)}(i) = 1, v^{(n)}_{k} = v)$ and $\bbP(A_k^{(\infty)}(i) = 1, v^{(\infty)}_{k} = v)$ (see \eqref{eq:dist_of_Akp-Vkp}) goes to zero. The difference of the second terms can be handled similarly as in the case $\ell = 0$; the situation is even simpler due to the absence of  the initial sum. This concludes the proof of this proposition.
\end{proof}

\begin{COR}\label{cor:three_cvg_results}
	Let $(P^{(n)})_{n \in \bbN}$, $P^{(n)} = (p_i^{(n)})_{i \in \bbN}$, be a sequence of probability distributions on $\bbN$ and let $P^{(\infty)} = (p_i^{(\infty)})_{i \in \bbN}$  be a probability distribution on $\bbN$. We assume that $g\equiv 1$. If $\,\norm[1]{P^{(n)} - P^{(\infty)}}_{\infty} = \sup_{i \in \bbN} \aps{p_i^{(n)} - p_i^{(\infty)}} \xrightarrow[n \rightarrow \infty]{} 0$, then for all $i \in \bbN$ and all fixed $k \in \bbN$	\begin{equation}\label{eq:cvg_in_d_of_A}
		A_k^{(P^{(n)})}(i) \xrightarrow[n \rightarrow \infty]{(d)} A_k^{(P^{(\infty)})}(i),
	\end{equation}
	\begin{equation}\label{eq:cvg_in_d_of_v}
		v_k^{(P^{(n)})} \xrightarrow[n \rightarrow \infty]{(d)} v_k^{(P^{(\infty)})},
	\end{equation}
	\begin{equation}\label{eq:cvg_in_d_of_Vp}
		V_k^{(P^{(n)})}(p_i^{(n)}) \xrightarrow[n \rightarrow \infty]{} V_k^{(P^{(\infty)})}(p_i^{(\infty)}),
	\end{equation}
	where $V_k^{(P^{(n)})}(p_i^{(n)}) = \bbE[A_k^{(P^{(n)})}(i) / v_k^{(P^{(n)})}]$, $n \in \bbN \cup \{\infty\}$, is the voting power of the node $i$ in the case $g\equiv 1$.
\end{COR}
\begin{proof}
Convergence in  \eqref{eq:cvg_in_d_of_A} and \eqref{eq:cvg_in_d_of_v} follows directly from Proposition \ref{prop:cvg_in_dist_of_(A,v)} using the continuous mapping theorem (see \cite[Theorem 3.2.4]{Durrett}) applied to projections $\FJADEF{\Pi_1, \Pi_2}{\bbR^2}{\bbR}$, $\Pi_i(x_1, x_2) = x_i$, $i = 1, 2$. 
To prove  the convergence in \eqref{eq:cvg_in_d_of_Vp}, let us first define the bounded and continuous function
	\begin{equation*}
		\FJADEF{\phi}{[0, \infty) \times [1, \infty)}{\bbR}, \quad \phi(x, y) = \min\VIT{\frac{x}{y}, 1}.
	\end{equation*}
Therefore, combining Proposition \ref{prop:cvg_in_dist_of_(A,v)} with \cite[Theorem 3.2.3]{Durrett} we get
	\begin{equation}\label{eq:exp_f_cvg}
		\bbE\UGL{\phi\OBL{A_k^{(P^{(n)})}(i), v_k^{(P^{(n)})}}} \xrightarrow[n \rightarrow \infty]{} \bbE\UGL{\phi\OBL{A_k^{(P^{(\infty)})}(i), v_k^{(P^{(\infty)})}}}.
	\end{equation}
	Notice that we always have $A_k^{(P)}(i) \le v_k^{(P)}$ since the random variable $A_k^{(P)}(i)$ counts the number of times the node $i$ was sampled until $k$ different nodes were sampled and the random variable $v_k^{(P)}$ counts the total number of samplings until $k$ distinct elements were sampled. Hence,
	\begin{equation*}
		\phi\OBL{A_k^{(P)}(i), v_k^{(P)}} = \frac{A_k^{(P)}(i)}{v_k^{(P)}}.
	\end{equation*}
	Combining the latter with \eqref{eq:exp_f_cvg} and using $V_k^{(P^{(n)})}(p_i^{(n)}) = \bbE[A_k^{(P^{(n)})}(i) / v_k^{(P^{(n)})}]$, $n \in \bbN \cup \{\infty\}$, we obtain \eqref{eq:cvg_in_d_of_Vp}.
\end{proof}

\section{Asymptotic fairness}\label{sec:Asymptotic_fairness}

We start this section with the case $k=2$, i.e., we sample until we get two different nodes. This small choice of $k$ allows us to perform analytical calculations and prove some facts rigorously. We prove that the voting scheme $(id, 1)$ is robust to merging but not fair. We also show that the more the node splits, the more voting power it can gain. However, with this procedure, the voting power does not grow to $1$, but a limit strictly less than $1$.

\begin{PROP}\label{prop:unfairness}
We consider the voting scheme $(id, 1)$ and let $(m_i)_{i \in \bbN}$ be the weight distribution of the nodes. Let $P = (p_i)_{i \in \bbN}$ be the corresponding probability distribution on $\bbN$, let $r \in \bbN,$ $i$ a node,  and $k = 2$. 
Then, for every $r$-splitting $m_{i_1^{(r)}}, \ldots, m_{i_r^{(r)}}>0$ of the node $i$, we have that 
	\begin{equation}\label{eq:diff_after-before_splitting}
	V_k^{(P)}(m_i) < \sum_{j = 1}^r V_k^{(\widehat{P}_{r,i})}(m_{i_j^{(r)}}).
			\end{equation}
In other words,  the voting scheme $(id, 1)$ is robust to merging, but not robust to splitting. The difference of the voting power after splitting and before splitting reaches its maximum for
	\begin{equation*}
		m_{i_j^{(r)}} = \frac{m_i}{r}, \qquad j \in \{1, 2, \ldots, r\}.
	\end{equation*}
	Furthermore, for this particular $r$-splitting, we have that the sequence
	\begin{equation*}
		\OBL{\Bigg( \sum_{j = 1}^r V_k^{(\widehat{P}_{r,i})}(m_{i_j^{(r)}}) \Bigg) - V_k^{(P)}(m_i)}_{r \in \bbN}
	\end{equation*}
	is strictly increasing and has a limit strictly less than $1$.
\end{PROP}
\begin{proof}
	Denote with $Y_i = A_2^{(P)}(i)$ the number of times that the node $i$ was sampled from the distribution $P$ until we sampled $2$ different nodes, and with $Y_{i_j^{}} = A_2^{(\widehat{P}_{r,i})}({i_j^{(r)}})$, $j \in \{1, 2, \ldots, r\}$, the number of times the node $i_j^{(r)}$ was sampled from the distribution $\widehat{P}_{r,i}$ until we sampled $2$ different nodes. We also write $V(m_i):= V_k^{(P)}(m_i)$ and $V^{(r)}(m_{i_j^{(r)}}):=V_k^{(\widehat{P}_{r,i})}(m_{i_j^{(r)}}).$ Using these notations, we have
	\begin{align*}
		V(m_i)
		& = \sum_{u \in \bbN \setminus \{i\}} \sum_{y_u = 1}^{\infty} \bbP(Y_i = 1, Y_u = y_u) \cdot \frac{1}{1 + y_u} + \sum_{u \in \bbN \setminus \{i\}} \sum_{y_i = 1}^{\infty} \bbP(Y_i = y_i, Y_u = 1) \cdot \frac{y_i}{1 + y_i} \\
		& = p_i \sum_{u \in \bbN \setminus \{i\}} \sum_{y_u = 1}^{\infty} p_u^{y_u} \cdot \frac{1}{1 + y_u} + \sum_{u \in \bbN \setminus \{i\}} p_u \sum_{y_i = 1}^{\infty} p_i^{y_i} \cdot \frac{y_i}{1 + y_i} \\
		& = p_i \sum_{u \in \bbN \setminus \{i\}} \OBL{\frac{-\log(1 - p_u)}{p_u} - 1} + (1 - p_i) \cdot \frac{\frac{p_i}{1 - p_i} + \log(1 - p_i)}{p_i} \\
		& = -p_i \sum_{u \in \bbN \setminus \{i\}} \OBL{\frac{\log(1 - p_u)}{p_u} + 1} + \frac{(1 - p_i) \log(1 - p_i)}{p_i} + 1.
	\end{align*}
	Similarly, for $j \in \{1, 2, \ldots, r\}$ we have
	\begin{align*}
		V^{(r)}(m_{i_j^{(r)}})
		& = \sum_{u \in \bbN \setminus \{i\}} \sum_{y_u = 1}^{\infty} \bbP(Y_{i_j^{}} = 1, Y_u = y_u) \cdot \frac{1}{1 + y_u}  \\
		& \qquad + \sum_{\substack{\ell = 1 \\ \ell \neq j}}^r \sum_{y_{i_\ell} = 1}^{\infty} \bbP(Y_{i_j^{}} = 1, Y_{i_l^{}} = y_{i_\ell}) \cdot \frac{1}{1 + y_{i_\ell}}  \\
		& \qquad + \sum_{u \in \bbN \setminus \{i\}} \sum_{y_{i_j} = 1}^{\infty} \bbP(Y_{i_j^{}} = y_{i_j}, Y_u = 1) \cdot \frac{y_{i_j}}{1 + y_{i_j}} \\
		& \qquad + \sum_{\substack{\ell = 1 \\ \ell \neq j}}^r \sum_{y_{i_j} = 1}^{\infty} \bbP(Y_{i_j^{}} = y_{i_j}, Y_{i_l^{}} = 1) \cdot \frac{y_{i_j}}{1 + y_{i_j}} \\
		& = -\widehat{p}_{i_j^{(r)}} \sum_{u \in \bbN \setminus \{i\}} \OBL{\frac{\log(1 - p_u)}{p_u} + 1} - \widehat{p}_{i_j^{(r)}} \sum_{\substack{\ell = 1 \\ \ell \neq j}}^r \OBL{\frac{\log(1 - \widehat{p}_{i_l^{(r)}})}{\widehat{p}_{i_l^{(r)}}} + 1} \\
		& \qquad + \frac{(1 - \widehat{p}_{i_j^{(r)}})\log(1 - \widehat{p}_{i_j^{(r)}})}{\widehat{p}_{i_j^{(r)}}} + 1.
	\end{align*}
	Combining the above calculations, we obtain
	\begin{align*}
		\sum_{j = 1}^r V^{(r)}(m_{i_j^{(r)}}) - V(m_i)
		& = -\sum_{j = 1}^r p_{i_j^{(r)}} \sum_{u \in \bbN \setminus \{i\}} \OBL{\frac{\log(1 - p_u)}{p_u} + 1} - \sum_{j = 1}^r \sum_{\substack{\ell = 1 \\ \ell \neq j}}^r \frac{\widehat{p}_{i_j^{(r)}} \log(1 - \widehat{p}_{i_l^{(r)}})}{\widehat{p}_{i_l^{(r)}}} \\
		& \qquad - \sum_{j = 1}^r p_{i_j^{(r)}} (r - 1) + \sum_{j = 1}^r \frac{(1 - \widehat{p}_{i_j^{(r)}})\log(1 - \widehat{p}_{i_j^{(r)}})}{\widehat{p}_{i_j^{(r)}}} + r \\
		&  \qquad	+ p_i \sum_{u \in \bbN \setminus \{i\}} \OBL{\frac{\log(1 - p_u)}{p_u} + 1} - \frac{(1 - p_i) \log(1 - p_i)}{p_i} - 1 \\
		& = \sum_{j = 1}^r \frac{\log(1 - \widehat{p}_{i_j^{(r)}})}{\widehat{p}_{i_j^{(r)}}} \OBL{1 - \widehat{p}_{i_j^{(r)}} - \sum_{\substack{\ell = 1 \\ \ell\neq j}}^r \widehat{p}_{i_\ell^{(r)}}} + r - 1 - p_i(r - 1) \\
		& \qquad - \frac{(1 - p_i) \log(1 - p_i)}{p_i} \\
		& = (1 - p_i) \UGL{(r - 1) + \sum_{j = 1}^r \frac{\log(1 - \widehat{p}_{i_j^{(r)}})}{\widehat{p}_{i_j^{(r)}}} - \frac{\log(1 - p_i)}{p_i}}.
	\end{align*}
	We take $x_1, x_2, \ldots x_r \in (0, 1)$ such that $\sum_{j = 1}^r x_j = 1$ and set
	\begin{equation*}
		\widehat{p}_{i_j^{(r)}} = p_i \cdot x_j, \quad j\in\{1,\ldots,r\}.
	\end{equation*}
	This gives us
	\begin{equation*}
		\sum_{j = 1}^r V^{(r)}(m_{i_j^{(r)}}) - V(m_i) = (1 - p_i) \UGL{(r - 1) + \sum_{j = 1}^r \frac{\log(1 - p_i x_j)}{p_i x_j} - \frac{\log(1 - p_i)}{p_i}}.
	\end{equation*}
	Define
	\begin{equation*}
		\phi(x_1, \ldots, x_r) := (r - 1) + \sum_{j = 1}^r \frac{\log(1 - p_i x_j)}{p_i x_j} - \frac{\log(1 - p_i)}{p_i}.
	\end{equation*}
	First,
we need to show that $\phi(x_1, x_2, \ldots, x_r) > 0$ for all $x_1, x_2, \ldots, x_r \in (0, 1)$ such that $\sum_{j = 1}^r x_j = 1$. Using Proposition \ref{prop:log_ineq} repeatedly ($r - 1$ times), we get
	\begin{align*}
		\phi(x_1, \ldots, x_r)
		& = (r - 2) + \OBL{1 + \frac{\log(1 - p_i x_1)}{p_i x_1} + \frac{\log(1 - p_i x_2)}{p_i x_2}} \\
		& \qquad + \sum_{j = 3}^r \frac{\log(1 - p_i x_j)}{p_i x_j} - \frac{\log(1 - p_i)}{p_i} \\
		& > (r - 2) + \frac{\log(1 - p_i(x_1 + x_2))}{p_i(x_1 + x_2)} + \sum_{j = 3}^r \frac{\log(1 - p_i x_j)}{p_i x_j} - \frac{\log(1 - p_i)}{p_i} \\
		& = (r - 3) + \OBL{1 + \frac{\log(1 - p_i(x_1 + x_2))}{p_i(x_1 + x_2)} + \frac{\log(1 - p_i x_3)}{p_i x_3}} \\
		& \qquad + \sum_{j = 4}^r \frac{\log(1 - p_i x_j)}{p_i x_j} - \frac{\log(1 - p_i)}{p_i} \\
		& > (r - 3) + \frac{\log(1 - p_i(x_1 + x_2 + x_3))}{p_i(x_1 + x_2 + x_3)} + \sum_{j = 4}^r \frac{\log(1 - p_i x_j)}{p_i x_j} - \frac{\log(1 - p_i)}{p_i} \\
		& \qquad \vdots\hspace{4cm} \vdots \hspace{4cm} \vdots \\
		& > 1 + \frac{\log(1 - p_i \sum_{j = 1}^{r - 1} x_j)}{p_i \sum_{j = 1}^{r - 1} x_j} + \frac{\log(1 - p_i x_r)}{p_i x_r} - \frac{\log(1 - p_i)}{p_i}  > 0.
	\end{align*}
The second claim  of this proposition is that the expression $$\sum_{j = 1}^r V^{(r)}(m_{i_j^{(r)}}) - V(m_i)$$ reaches its maximum for $\widehat{p}_{i_j^{(r)}} = \frac{p_i}{r}, j \in \{1, 2, \ldots, r\}$. This follows directly from Lemma \ref{lem:Lagrange_multipliers}, where we show that $\phi$ attains its unique maximum for $(x_1, \ldots, x_r) = (\frac{1}{r}, \ldots, \frac{1}{r})$. Denote with
	\begin{align*}
		\tau_r(p_i) = (1 - p_i) \phi \OBL{\frac{1}{r}, \ldots, \frac{1}{r}}
		& = (1 - p_i) \UGL{(r - 1) + \sum_{j = 1}^r \frac{\log(1 - \frac{p_i}{r})}{\frac{p_i}{r}} - \frac{\log(1 - p_i)}{p_i}} \\
		& = (1 - p_i) \UGL{r + \frac{r^2 \log(1 - \frac{p_i}{r})}{p_i} - \frac{\log(1 - p_i)}{p_i} - 1}.
	\end{align*}
By Proposition \ref{prop:tau_r_is_inc} we have that the sequence $(\tau_r(p_i))_{r \in \bbN}$ is strictly increasing and
	\begin{equation*}
		\tau(p_i) = \lim_{r \to \infty} \tau_r(p_i) = (1 - p_i) \OBL{-\frac{p_i}{2} - \frac{\log(1 - p_i)}{p_i} - 1}.
	\end{equation*}
\nopagebreak[4]
\end{proof}

\begin{REM}
	We consider the  function $\FJADEF{\tau}{(0, 1)}{\bbR}$ defined by
	\begin{equation*}
		\tau(m) = (1 - m) \OBL{-\frac{m}{2} - \frac{\log(1 - m)}{m} - 1}.
	\end{equation*}
This function describes the gain in voting power a node with initial weight $m$ can achieve by splitting up into infinitely many nodes. As  Figure \ref{fig:graph_of_function_tau} shows, this maximal gain in voting power is bounded.  The function $\tau$ attains maximum at $m^* \approx 0.82$ and the maximum is $\tau(m^*) \approx 0.12$. This means that a node that initially has around $82\%$ of the total amount of mana can obtain the biggest gain in the voting power (by theoretically splitting into infinite number of nodes) and this gain is approximately $0.12$. Loosely speaking, if a voting power of a node increases by $0.12$, this means that during the querying, the proportion of queries that are addressed to this particular node increases by around $12\%$.
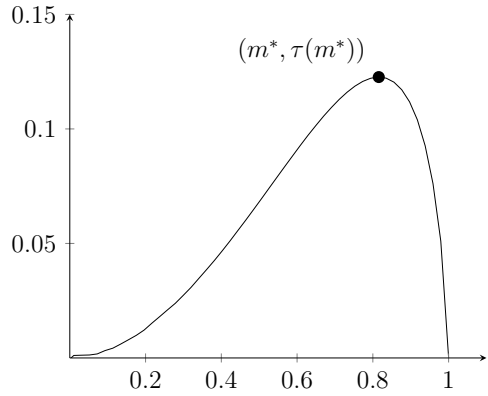
\begin{figure}[h!]
	\centering
	\begin{tikzpicture}[scale = 0.8]
		\begin{axis}[
			ymin = 0,
			ymax = 0.15,
			xmin = 0,
			xmax = 1.1,
			axis on top = true,
			axis x line = middle,
			axis y line = middle,
			ytick={0.05,0.1,0.15},
			yticklabels={$0.05$, $0.1$,$0.15$},
			]
			\addplot [samples = 50,
				domain = -0.01:1]
				{(1 - x) * (-x/2 - ln(1 - x)/x - 1)};
			\node[label = {130:{$(m^*, \tau(m^*))$}}, circle, fill, inner sep = 2pt] at (axis cs:0.815658,0.122642) {};
		\end{axis}
	\end{tikzpicture}
	\caption{Graph of the function $\tau$.}\label{fig:graph_of_function_tau}
\end{figure}
\end{REM}

\begin{COR}\label{cor:asymp_unfairness}
Let $m^{(n)}$ be a sequence of weight distributions with corresponding probability distributions $(P^{(n)})_{n \in \bbN}$, $P^{(n)} = (p_i^{(n)})_{i \in \bbN}$  on $\bbN$. Let $m^{(\infty)}$ be a weight distribution such that for its corresponding probability distributions  $P^{(\infty)} = (p_i^{(\infty)})_{i \in \bbN}$ we have that  $$\norm[1]{P^{(n)} - P^{(\infty)}}_{\infty} = \sup_{i \in \bbN} \aps{p_i^{(n)} - p_i^{(\infty)}} \xrightarrow[n \rightarrow \infty]{} 0.$$ Furthermore, we consider a sequence of $r$-splittings  $m_{i_1^{(r)}}^{(n)}, \ldots, m_{i_r^{(r)}}^{(n)} > 0$ of a node $i$ such that  $ m_{i_j^{(r)}}^{(n)} \xrightarrow[n \rightarrow \infty]{}  m_{i_j^{(r)}}^{(\infty)}$, $j \in \{1, 2, \ldots, r\}$, for some $r$-splitting $m^{(\infty)}$. Then, for $k = 2$ we have
	\begin{align*}
		\lim_{n \to \infty}
		& \OBL{\Bigg( \sum_{j = 1}^r V_k^{(\widehat P_{r,i}^{(n)})}(\widehat p_{i_j^{(r)}}^{(n)}) \Bigg) - V_k^{(P^{(n)})}(p_i^{(n)})} \\
		& = \Bigg( \sum_{j = 1}^r V_k^{(\widehat P_{r,i}^{(\infty)})}(\widehat p_{i_j^{(r)}}^{(\infty)}) \Bigg) - V_k^{(P^{(\infty)})}(p_i^{(\infty)})  > 0.
	\end{align*}
\end{COR}
\begin{proof} The convergence follows directly from Corollary \ref{cor:three_cvg_results}, and the strict positivity of the limit follows from Proposition \ref{prop:unfairness}.
\end{proof}

\begin{REM}\label{rem:asymp_unfairness}
 Corollary \ref{cor:asymp_unfairness} implies that if $k=2$  and if the sequence of weight distributions $(P^{(n)})_{n \in \bbN}$ converges to a non-trivial probability distribution on $\bbN$, the voting scheme $(id, 1)$ is not asymptotically fair. Applying this result to the sequence of Zipf distributions defined in \eqref{eq:seq_of_Zipf_law}, we see that for $s > 1$, and $k = 2$, the voting scheme is not  asymptotically fair. Simulations suggest, see Figures \ref{fig:Asymptotic_fairness} and \ref{fig:incMeank}, that for higher values of  $k$ the difference in voting power of the node $i$ before and after the splitting does not converge to zero as the number of nodes in the network grows to infinity.
\end{REM}

In the following proposition we give a condition on the sequence of weight distributions $(P^{(n)})_{n \in \bbN}$ under which the voting scheme $(id, 1)$ is asymptotically fair for any choice of  the parameter $k$.

\begin{TM}\label{thm:asymptotic_fairness}
Let $k$ be  a positive integer and  $m^{(n)}$ be a sequence of weight distributions with corresponding probability distributions $(P^{(n)})_{n \in \bbN}$, $P^{(n)} = (p_i^{(n)})_{i \in \bbN}$  on $\bbN$.  We assume that  $\norm[1]{P^{(n)}}_{\infty} \xrightarrow[n \rightarrow \infty]{} 0$.  Furthermore, we consider a sequence of $r$-splittings  $m_{i_1^{(r)}}^{(n)}, \ldots, m_{i_r^{(r)}}^{(n)} > 0$ of a given node $i$ such that  $ m_{i_j^{(r)}}^{(n)} \xrightarrow[n \rightarrow \infty]{}  m_{i_j^{(r)}}^{(\infty)}$, $j \in \{1, 2, \ldots, r\}$, for some $r$-splitting $m^{(\infty)}$. 
Then, 
	\begin{equation*}
		\APS{\Bigg( \sum_{j = 1}^r V_k^{(\widehat P_{r,i}^{(n)})}(\widehat p_{i_j^{(r)}}^{(n)}) \Bigg) - V_k^{(P^{(n)})}(p_i^{(n)})} \xrightarrow[n \rightarrow \infty]{} 0,
	\end{equation*}
	i.e., the voting scheme $(id, 1)$ is asymptotically fair if the sequence of weight distributions converges in the supremum norm to $0$.
	\end{TM}
\begin{proof}
	For simplicity, we write $P_r^{(n)} := \widehat P_{r,i}^{(n)}$, $v_k^{(n)} := v_k^{(P^{(n)})}$ and $v_{k, r}^{(n)} := v_k^{(P_r^{(n)})}$; recall that the random variable $v_k^{(P)}$ counts the number of samplings with replacement from the distribution $P$ until $k$ different elements are sampled. The main idea of the proof is to couple the random variables $v_k^{(n)}$ and $v_{k, r}^{(n)}$. We sample simultaneously from probability distributions $P^{(n)}$ and $P_r^{(n)}$ and construct two different sequences of elements that both terminate once they contain $k$ different elements. We do that in the following way: we sample an element from the distribution $P^{(n)}$. If the sampled element is not $i$, we just add this element to both sequences that we are constructing and then sample the next element. If the element $i$ is sampled, then we add $i$ to the first sequence, but to the second sequence we add one of the elements $i_1^{(r)}, \ldots, i_r^{(r)}$ according to the probability distribution $(p_{i_1^{(r)}}^{(n)} / p_i^{(n)}, \ldots, p_{i_r^{(r)}}^{(n)} / p_i^{(n)})$. Now,  the second sequence will terminate not later than  the first one since the second sequence always has at least the same amount of different elements as the first sequence. This is a consequence of the fact that, each time the element $i$ is sampled, we add one of the $r$ elements $i_1^{(r)}, i_2^{(r)}, \ldots, i_r^{(r)}$ to the second sequence while we just add $i$  to the first sequence, see Figure \ref{fig:Coupling}.
		\begin{figure}[h!]
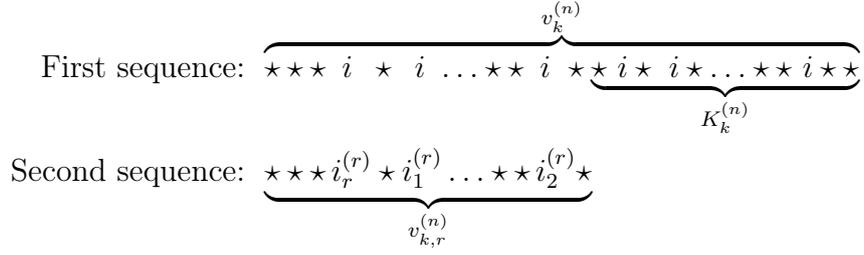

			\begin{align*}
				\textnormal{First sequence: }
				& \overbrace{\star \star \star \,\,\,i\,\,\, \star \,\,\,i\,\, \ldots \star \star \,\,\,i\,\, \star
				\underbrace{\star \,\,i \star \,\,i \star \ldots \star \star \,\,i \star \star}_{K_k^{(n)}}}^{v_k^{(n)}} \\
				\textnormal{Second sequence: }
				& \underbrace{\star \star \star\,   i_r^{(r)} \star i_1^{(r)} \ldots \star \star \, i_2^{(r)} \star}_{v_{k, r}^{(n)}}
			\end{align*}
			\caption{Coupling of random variables $v_k^{(n)}$ and $v_{k, r}^{(n)}$.}
			\label{fig:Coupling}
		\end{figure}

	Denote with
	\begin{equation*}
		K_k^{(n)} := v_k^{(n)} - v_{k, r}^{(n)}.
	\end{equation*}
	Since $v_k^{(n)} \ge v_{k, r}^{(n)}$, we have $K_k^{(n)} \ge 0$. We also introduce the random variable
	\begin{equation*}
		L_k^{(n)} := A_k^{(P^{(n)})}(i) - \sum_{j = 1}^r A_k^{(P_r^{(n)})}(i_j^{(r)}),
	\end{equation*}		
	where $A_k^{(P)}(i)$ is defined as in \eqref{eq:def_of_AkP}. The random variable $L_k^{(n)}$ is measuring the difference in the number of times the node $i$ appears in the first sequence and the number of times nodes $i_1^{(r)}, i_2^{(r)}, \ldots, i_r^{(r)}$ appear in the second sequence. At the time when the second sequence terminates, the number of times the node $i$ appeared in the first sequence is the same as the number of times that nodes $i_1^{(r)}, i_2^{(r)}, \ldots, i_r^{(r)}$ appeared in the second sequence, see Figure \ref{fig:Coupling}. Since the length of the first sequence is always larger than or equal to the length of the second sequence, it can happen that the element $i$ is sampled again before the $k$-th different element appears in the first sequence. Therefore, $L_k^{(n)} \ge 0$. Clearly, $L_k^{(n)} \le K_k^{(n)}$ because $K_k^{(n)}$ counts all the extra samplings we need to sample $k$ different elements in the first sequence, while $L_k^{(n)}$ counts only those extra samplings in which the element $i$ was sampled. Notice that if the element $i$ is not sampled before the $k$-th different element appears or if $i$ is the $k$-th different element, then $K_k^{(n)} = L_k^{(n)} = 0$.
	
	Let
	\begin{equation*}
		Y_k^{(n)} := A_k^{(P^{(n)})}(i) \quad \textnormal{and} \quad Y_{k, r}^{(n)} := \sum_{j = 1}^r A_k^{(P_r^{(n)})}(i_j^{(r)}).
	\end{equation*}
	Then
	\begin{equation*}
		V_k^{(P^{(n)})}(p_i^{(n)}) = \bbE\UGL{\frac{Y_k^{(n)}}{v_k^{(n)}}}, \quad \sum_{j = 1}^r V_k^{(P_r^{(n)})}(p_{i_j^{(r)}}^{(n)}) = \bbE\UGL{\frac{Y_{k, r}^{(n)}}{v_{k, r}^{(n)}}} \quad \textnormal{and} \quad Y_k^{(n)} = Y_{k, r}^{(n)} + L_k^{(n)}.
	\end{equation*}
Combining this with $v_k^{(n)} = v_{k, r}^{(n)} + K_k^{(n)}$, we have
	\begin{align*}
		\Bigg\vert
		& \Bigg( \sum_{j = 1}^r V_k^{(P_r^{(n)})}(p_{i_j^{(r)}}^{(n)}) \Bigg) - V_k^{(P^{(n)})}(p_i^{(n)})\Bigg\vert = \APS{\bbE\UGL{\frac{Y_{k, r}^{(n)}}{v_{k, r}^{(n)}} - \frac{Y_k^{(n)}}{v_k^{(n)}}}} = \APS{\bbE\UGL{\frac{Y_{k, r}^{(n)}}{v_{k, r}^{(n)}} - \frac{Y_{k, r}^{(n)} + L_k^{(n)}}{v_{k, r}^{(n)} + K_k^{(n)}}}} \\
		& = \APS{\bbE\UGL{\frac{Y_{k, r}^{(n)} (v_{k, r}^{(n)} + K_k^{(n)}) - (Y_{k, r}^{(n)} + L_k^{(n)}) v_{k, r}^{(n)}}{v_{k, r}^{(n)} (v_{k, r}^{(n)} + K_k^{(n)})}}} = \APS{\bbE\UGL{\frac{Y_{k, r}^{(n)} K_k^{(n)}}{v_{k, r}^{(n)} (v_{k, r}^{(n)} + K_k^{(n)})} - \frac{L_k^{(n)}}{v_{k, r}^{(n)} + K_k^{(n)}}}} \\
		& \le \bbE\UGL{\frac{Y_{k, r}^{(n)}}{v_{k, r}^{(n)}} \cdot \frac{1}{v_{k, r}^{(n)} + K_k^{(n)}} \cdot K_k^{(n)} + \frac{1}{v_{k, r}^{(n)} + K_k^{(n)}} \cdot L_k^{(n)}}.
	\end{align*}
	Denote with
	\begin{equation*}
		Z_n := \frac{Y_{k, r}^{(n)}}{v_{k, r}^{(n)}} \cdot \frac{1}{v_{k, r}^{(n)} + K_k^{(n)}} \cdot K_k^{(n)} + \frac{1}{v_{k, r}^{(n)} + K_k^{(n)}} \cdot L_k^{(n)}.
	\end{equation*}
It remains to prove  that $\bbE[Z_n] \xrightarrow[n \rightarrow \infty]{} 0$. Since $Y_{k, r}^{(n)} \le v_{k, r}^{(n)}$, $k \le v_{k, r}^{(n)},$ and $L_k^{(n)} \le K_k^{(n)}$ we have
	\begin{equation*}
		Z_n \le \min\{2K_k^{(n)}, 2\}.
	\end{equation*}
	By Lemma \ref{lm:v_k^n->k} we have that $v_k^{(n)} \xrightarrow[n \rightarrow \infty]{\bbP} k$ and $v_{k, r}^{(n)} \xrightarrow[n \rightarrow \infty]{\bbP} k$ (notice that $\norm[1]{P_r^{(n)}}_{\infty} \le \norm[1]{P^{(n)}}_{\infty}$). Therefore,
	\begin{equation*}
		K_k^{(n)} = v_k^{(n)} - v_{k, r}^{(n)} \xrightarrow[n \rightarrow \infty]{\bbP} 0.
	\end{equation*}
	This implies that $Z_n \xrightarrow[n \rightarrow \infty]{\bbP} 0$. Since $Z_{n} \leq 2$ we have that $\lim_{n \to \infty} \bbE[Z_n] = 0$.
\end{proof}

\begin{REM}\label{rem:comment_on_af}
	The above proposition shows that $\norm[1]{P^{(n)}}_{\infty} \xrightarrow[n \rightarrow \infty]{} 0$ is a sufficient condition to ensure asymptotic fairness, regardless of the value of the parameter $k \in \bbN$. Applying this result to the sequence of probability distributions $(P^{(n)})_{n \in \bbN}$ defined by the Zipf's law (see \eqref{eq:seq_of_Zipf_law}) we see that for $s \le 1$ the voting scheme $(id,1)$ is asymptotically fair.
\end{REM}

\section{Simulations and conjectures}\label{sec:Simulations}
In this section, we present some numerical simulations to complement our theoretical results. We are interested in the rate of convergence in the asymptotic fairness, Theorem \ref{thm:asymptotic_fairness}, and want to support some conjectures for the situation where our theoretical results do not apply.

We always consider a Zipf law for the nodes' weight distribution; see Relation \eqref{eq:seq_of_Zipf_law}. The reasons for this assumption are presented in Subsection \ref{subsec:Zipf_law}. We always consider the voting scheme  $(id, 1)$.

Figure \ref{fig:Asymptotic_fairness} presents results of a Monte-Carlo simulation for a Zipf distribution with parameter $s\in\{0.8, 1.1\}$ and different network sizes on the $x$-axis.  For real-world applications we expect values of $k$ to be at least $20$, see also \cite{manaFPC}, and set, therefore, the sample size to $k = 20$. The $y$-axis shows the gain in voting power for the heaviest node splitting into two nodes of equal weight.  For each choice of network size, we performed $1\,000\,000$ simulations and use the empirical average as an estimator for the gain in voting power. The gray zone corresponds to the confidence interval of level $0.95$. Let us note that to decrease the variance of the estimation, we couple, as in the proof of Theorem \ref{thm:asymptotic_fairness}, the sampling in the original network with the sampling in the network after splitting. 

Theorem \ref{thm:asymptotic_fairness} and Remark \ref{rem:comment_on_af} state that if the Zipf parameter $s\le 1$ the voting scheme is asymptotically fair, i.e., the difference of the voting power after the splitting and before the splitting of a node $i \in \bbN$ goes to zero as the number of nodes in the network increases.  The left-hand side of Figure \ref{fig:Asymptotic_fairness}  indicates the speed of convergence for $s=0.8.$ 
\begin{figure}[ht]
	\centering
	\includegraphics[scale=0.44]{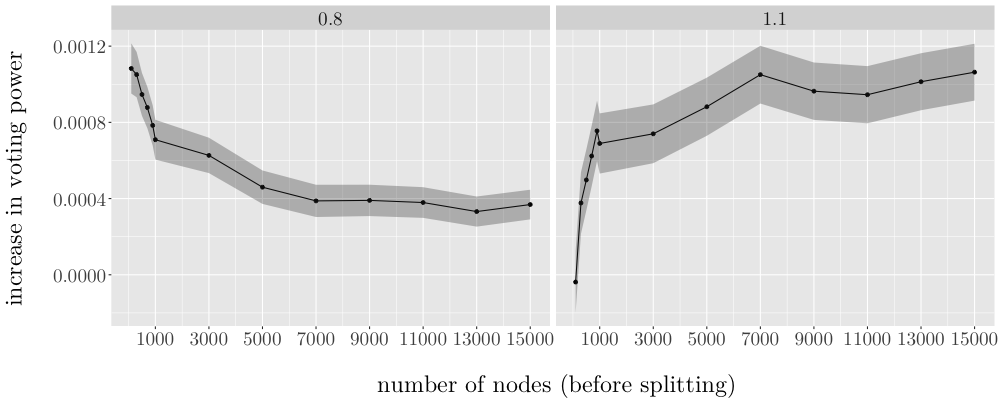}
	\caption{Increase in voting power for fixed $k = 20$, varying $N$, and two different values of $s$.}	\label{fig:Asymptotic_fairness}
\end{figure}
The right-hand side of Figure \ref{fig:Asymptotic_fairness} indicates that for $s=1.1$ the voting scheme is not asymptotically fair. Corollary \ref{cor:asymp_unfairness} states that for $k=2$, if the sequence of weight distributions $(P^{(n)})_{n \in \bbN}$ converges to a non-trivial probability distribution on $\bbN$, the voting scheme $(id, 1)$ is not asymptotically fair. 
\begin{CON}\label{con:asymp_unfairness}
Let $m^{(n)}$ be a sequence of weight distributions with corresponding probability distributions $(P^{(n)})_{n \in \bbN}$, $P^{(n)} = (p_i^{(n)})_{i \in \bbN}$  on $\bbN$. Let $m^{(\infty)}$ be a weight distribution such that for its corresponding probability distribution  $P^{(\infty)} = (p_i^{(\infty)})_{i \in \bbN}$ we have that  $$\norm[1]{P^{(n)} - P^{(\infty)}}_{\infty} = \sup_{i \in \bbN} \aps{p_i^{(n)} - p_i^{(\infty)}} \xrightarrow[n \rightarrow \infty]{} 0.$$ Furthermore, we consider a sequence of $r$-splittings  $m_{i_1^{(r)}}^{(n)}, \ldots, m_{i_r^{(r)}}^{(n)} > 0$ of a node $i$ such that  $ m_{i_j^{(r)}}^{(n)} \xrightarrow[n \rightarrow \infty]{}  m_{i_j^{(r)}}^{(\infty)}$, $j \in \{1, 2, \ldots, r\}$, for some $r$-splitting $m^{(\infty)}$. Then, for any choice of $k\in \bbN$
	\begin{align*}
		\lim_{n \to \infty}
		& \OBL{\Bigg( \sum_{j = 1}^r V_k^{(\widehat P_{r,i}^{(n)})}(\widehat p_{i_j^{(r)}}^{(n)}) \Bigg) - V_k^{(P^{(n)})}(p_i^{(n)})} \\
		& = \Bigg( \sum_{j = 1}^r V_k^{(\widehat P_{r,i}^{(\infty)})}(\widehat p_{i_j^{(r)}}^{(\infty)}) \Bigg) - V_k^{(P^{(\infty)})}(p_i^{(\infty)})  > 0.
	\end{align*}
\end{CON}

We take a closer look at the distribution of the increase in voting power in the above setting. Figures \ref{fig:incDensity11} and \ref{fig:incDensity08}  present density estimations, with a gaussian kernel, of the density of the increase in voting power. Again we simulated each data point $1\,000\,000$ times. The density's multimodality should be explained by the different possibilities the heaviest node before and after splitting can be chosen. Figure  \ref{fig:incDensity08} explains well the asymptotic fairness; the probability of having only a small change in voting power converges to $0$ as the number of participants grows to infinity. Figure  \ref{fig:incDensityCompared} compares the densities for different choices of $s$ in a network of $1000$ nodes. 

\begin{figure}[ht]
	\centering
	\includegraphics[scale=0.4]{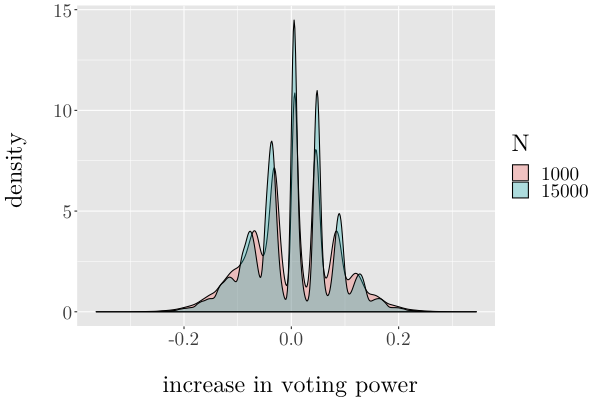}
	\caption{Density estimation for increase in voting power for two choices of network sizes. ($k = 20, s = 1.1$).}	\label{fig:incDensity11}
\end{figure}

\begin{figure}[ht]
	\centering
	\includegraphics[scale=0.4]{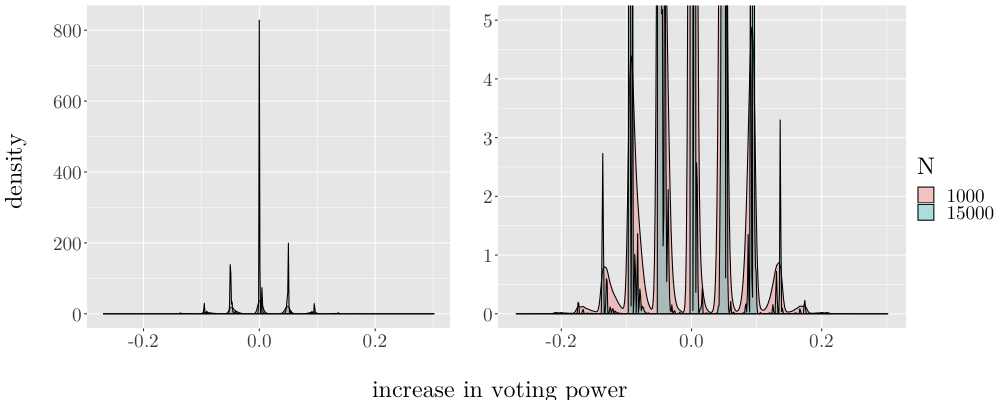}
	\caption{Density estimation for increase in voting power for two choices of network sizes. ($k = 20, s = 0.8$). The right-hand side is a zoom of the left-hand side.}	\label{fig:incDensity08}
\end{figure}

\begin{figure}[ht]
	\centering
	\includegraphics[scale=0.4]{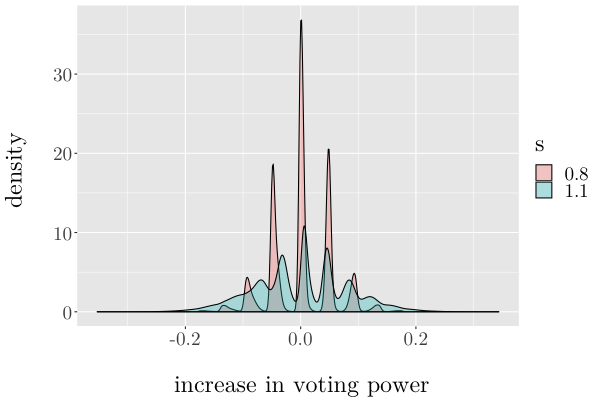}
	\caption{Density estimation for increase in voting power for two choices of the Zipf parameter $s$ for a network size of $1000$ nodes and $k = 20$.}	\label{fig:incDensityCompared}
\end{figure}

The last figures also show that even in the case where a splitting leads to an increase on average of the voting power, the splitting can also lead to less influence in a single voting round.

We kept the sample size $k=20$ in the previous simulations. Increasing the sample size increases the quality of the voting, however with the price of a higher message complexity. Figure \ref{fig:incMeank} compares the increase in voting power for different values of $k$ and $s$. We can see that an increase of $k$ increases the fairness of the voting scheme and that for some values of $k$ the increase in voting power may even be negative.

\begin{figure}[ht]
	\centering
	\includegraphics[scale=0.4]{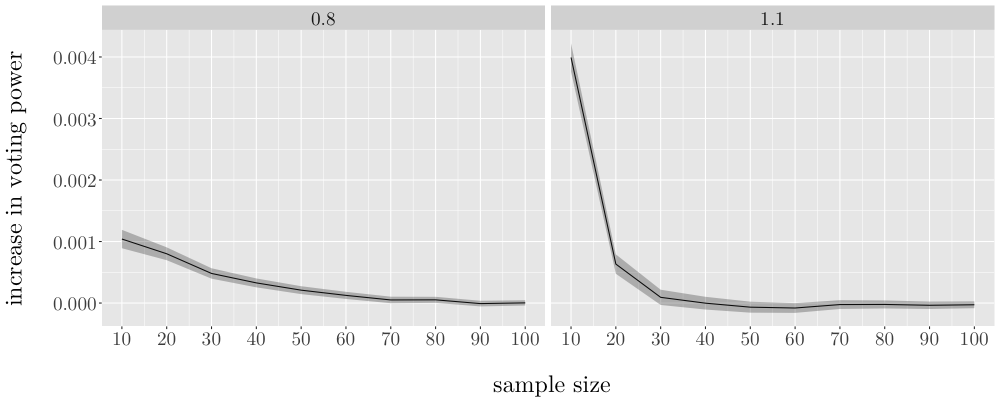}
	\caption{Mean increase in voting power for different values of the sample size $k$, two different values of $s$, and  a network size of $1000$ nodes.}	\label{fig:incMeank}
\end{figure}
Figure \ref{fig:incDensityComparedk} presents density estimations of the increase of voting power. 
\begin{figure}[ht]
	\centering
	\includegraphics[scale=0.4]{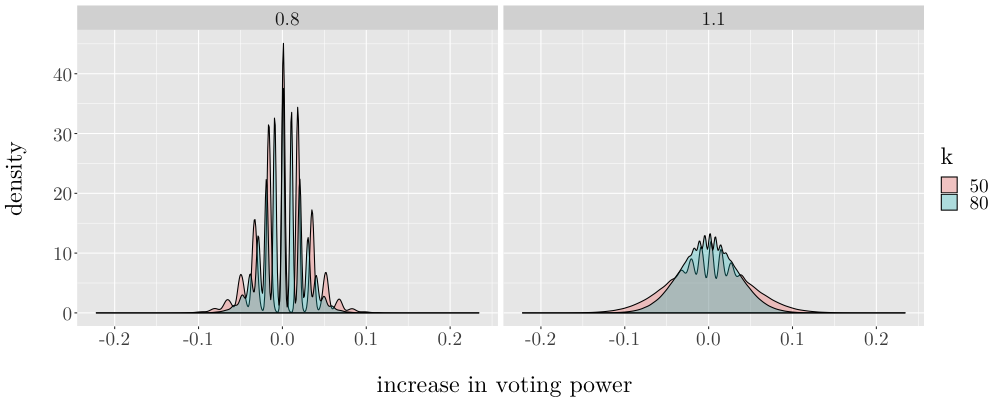}
	\caption{Density estimation for increase in voting power for two choices of the Zipf parameter $s$ for a network size of $1000$ nodes and two choices of $k$.}	\label{fig:incDensityComparedk}
\end{figure}
We can see the different behaviors in the more decentralized setting, $s<1$, and the centralized setting, $s>1$.  In the first case, it seems that the density converges to a point mass in $0$, whereas in the second case, the limit may be described by a Gaussian density.  A QQ-plot supports this first visual impression in Figure \ref{fig:qqplot}.
\begin{figure}[ht]
	\centering
	\includegraphics[scale=0.5]{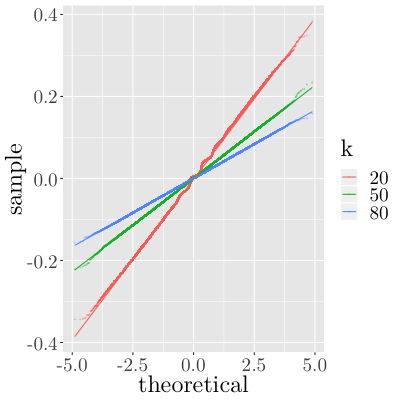}
	\caption{QQ-plots of the increase in voting power against a Gaussian distribution for different choices of $k$ and $s=1.1$ in a network of size $1000$.}	\label{fig:qqplot}
\end{figure}

While the study of the actual distribution of the increase in voting power is out of the scope of this paper we think that the following questions might be of interest.
\begin{QUE}
In what way can the distribution of the increase in voting power be described? 
\end{QUE}
\begin{QUE}
What kind of characteristics of the distribution of the increase in voting power are important for the voting scheme and its applications.
\end{QUE}
Recall that we only considered the change in voting power of the heaviest node that splits into two nodes of equal weight until now.

The goal of the next two simulations, see Figures \ref{fig:Effect_of_splitting_on_voting_power-fixed_k} and \ref{fig:Effect_of_splitting_on_voting_power-different_k_and_s}, is to inspect what happens with the voting power of a node when it splits into more than just two nodes.

For the simulation shown in Figure \ref{fig:Effect_of_splitting_on_voting_power-fixed_k}, we fix the value of the parameter $k$ and we vary the value of the parameter $s$. In Proposition \ref{prop:unfairness} we showed that for $k=2$, a node always gains voting power with splitting. This result holds without any additional assumptions on the weight distribution of the nodes in the network.  We run  simulations with $k = 20$ and we split the heaviest node into $r$ nodes; $r$ ranging from $2$ to $200$.\begin{figure}[ht]
	\centering
	\includegraphics[scale=0.5]{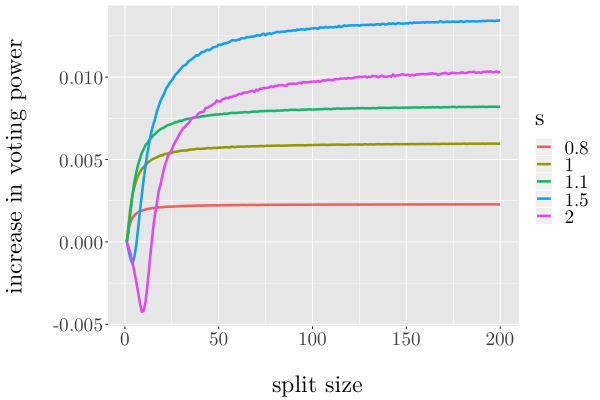}
	\caption{The effect of multiple splitting on the voting power of a node for $k = 20$ and different $s$ in a network of size $1000$.}\label{fig:Effect_of_splitting_on_voting_power-fixed_k}
\end{figure}
We keep  the network size equal to $1000$ and vary the parameter $s$ in the set $\{0.8, 1, 1.1, 1.5, 2\}$. For each different value of the parameter $s$ we ran $100\,000$ simulations of the voting scheme $(id, 1)$. Several conjectures can be made from Figure \ref{fig:Effect_of_splitting_on_voting_power-fixed_k}. It seems that if the parameter $k$ is equal to $20$, we can even have a drop in the voting power for small values of the parameter $r$. This drop appears to be more significant the bigger the parameter $s$ is. But if we split into more nodes (we set $r$ to be sufficiently high),  it seems that splitting gives us more voting power, and the gain is bigger for values of $s$ larger than $1$. This suggests that it is possible to have robustness to splitting into $r$ nodes for $r$ smaller than some threshold $\delta$, and robustness to merging of $r$ nodes for $r > \delta$.

The simulations presented in Figure \ref{fig:Effect_of_splitting_on_voting_power-different_k_and_s} show the change of the voting power of a node after it splits into multiple nodes for different values of the parameters $k$ and $s$.  As in the previous simulation, we consider a network size of $1000$ and assume that the first node splits into $r$ different nodes (where $r$ is again ranging from $2$ to $200$). For each combination of values of parameters $k$ and $s$, we ran $100\,000$ simulations. Our results suggest that for $s \le 1$, we always gain voting power with additional splittings.
On the other hand, if $s > 1$ then the voting power's behavior depends even more on the precise value of $k$. It seems that for small $k$, we still cannot lose voting power by splitting, but for $k$ sufficiently large it seems that there is a region where the increase in voting power is negative. 

\begin{figure}[ht]
	\centering
	\includegraphics[scale=0.45]{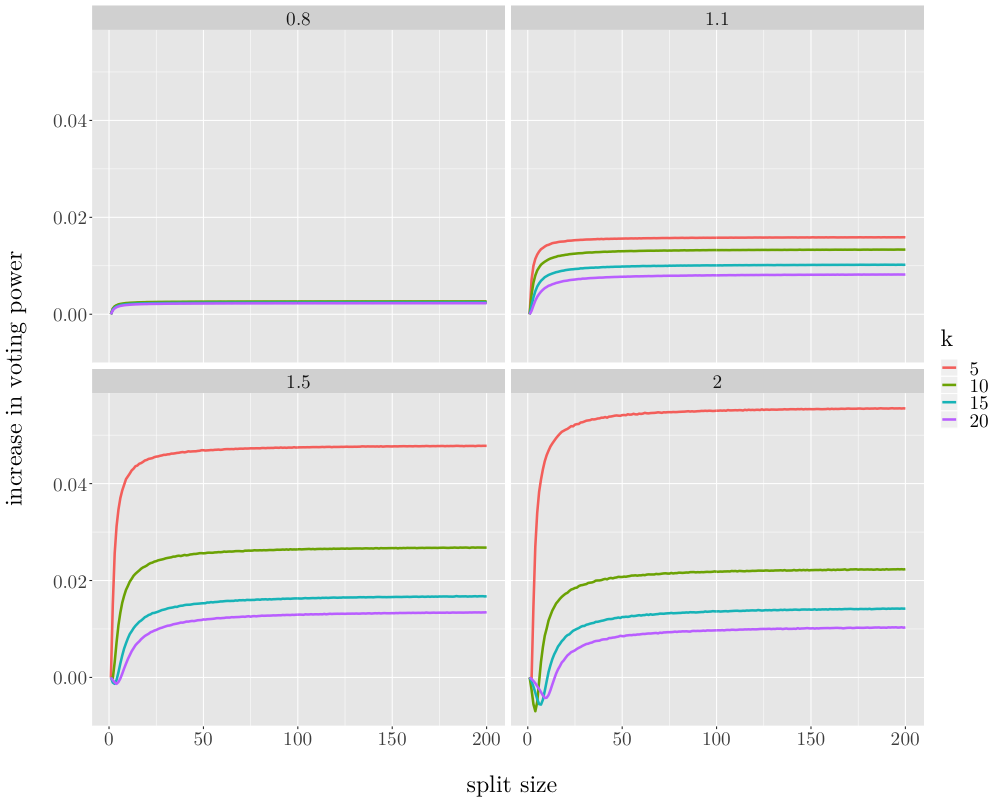}
	\caption{The effect of multiple splitting on the voting power of the heaviest node in a network of size $1000$ ($s \in \{0.8, 1.1, 1.5, 2\}$, $k \in \{5, 10, 15, 20\}$).}\label{fig:Effect_of_splitting_on_voting_power-different_k_and_s}
\end{figure}

\begin{QUE}
How does the increase in voting power of the heaviest node depends on $k$, $s$, and $N$? For which choices of these parameters the increase in voting power is negative?
\end{QUE}

The above simulation study is far from complete, but we believe that our results already show the model's richness. In the simulations, we only split the heaviest node.
\begin{QUE}
How does the increase in voting power of the node of rank $M$ depends on $M$, $k$, $s$, and $N$?
\end{QUE}

In a more realistic model, not only one but all nodes may simultaneously optimize their voting power. This is particularly interesting in situations that are not robust to splitting. We believe that it is reasonable that nodes may adapt their strategy from time to time to optimize their voting power in such a situation. This simultaneous splitting or merging of the nodes may lead to a periodic behavior of the nodes or convergence to a stable situation, where none of the nodes has an incentive to split or merge.

\begin{QUE}
Construct a multi-player game where the aim is to maximize the voting power. Do the corresponding weights always converge to a situation in which the voting scheme is fair?
\end{QUE}

\section{Appendix}\label{sec:Appendix}

In this section, we provide proofs of several results that we use throughout the paper.
\subsection{Auxiliary results for Section \ref{sec:Dist_of_sample_size}}

The first result is proved by induction.
\begin{LM}\label{lm:telescope_equality}
	Let $w \in \bbN$ and let $a_1, a_2, \ldots, a_w, b_1, b_2, \ldots, b_w \in \bbR$, then
	\begin{equation*}
		\left(a_1 a_2 \cdots a_w\right) - \left(b_1 b_2 \cdots b_w \right)= \sum_{j = 1}^w a_1 a_2 \cdots a_{j - 1} (a_j - b_j) b_{j + 1} \cdots b_w.
	\end{equation*}
\end{LM}
\begin{PROP}\label{prop:equiv_of_cvg_in_l-1_and_l-infty}
	Let $(P^{(n)})_{n \in \bbN}$, $P^{(n)} = (p_i^{(n)})_{i \in \bbN}$, be a sequence of probability distributions on $\bbN$, and let $P^{(\infty)} = (p_i^{(\infty)})_{i \in \bbN}$ be  a probability distribution on $\bbN$. Then, the following statements are equivalent:
	\begin{enumerate}[(a)]
		\item $\norm[1]{P^{(n)} - P^{(\infty)}}_{\infty} = \sup_{i \in \bbN} \aps{p_i^{(n)} - p_i^{(\infty)}} \xrightarrow[n \rightarrow \infty]{} 0$,
		\item $\norm[1]{P^{(n)} - P^{(\infty)}}_1 = \sum_{i = 1}^{\infty} \aps{p_i^{(n)} - p_i^{(\infty)}} \xrightarrow[n \rightarrow \infty]{} 0$.
	\end{enumerate}
\end{PROP}
\begin{proof}
	(b) $\Rightarrow$ (a): This follows immediately from $\sup_{i \in \bbN} \aps{p_i^{(n)} - p_i^{(\infty)}} \le \sum_{i = 1}^{\infty} \aps{p_i^{(n)} - p_i^{(\infty)}}$.
	
\noindent
(a) $\Rightarrow$ (b): Let $\varepsilon > 0$. Choose $n_0 = n_0(\varepsilon)$ such that
	\begin{equation}\label{eq:tail_of_p^inf}
		\sum_{i = 1}^{n_0} p_i^{(\infty)} > 1 - \varepsilon.
	\end{equation}
	This can be done because $P^{(\infty)}$ is a probability distribution on $\bbN$. Furthermore, let $n_1 = n_1(\varepsilon, n_0) \in \bbN$ be such that for every $n \ge n_1$ we have
	\begin{equation}\label{eq:l^inf_bound}
		\norm[1]{P^{(n)} - P^{(\infty)}}_{\infty} < \frac{\varepsilon}{n_0}.
	\end{equation}
	Using this we get for all $n \ge n_1$:	\begin{equation}\label{eq:start_of_l^1_bound}
		\sum_{i = 1}^{n_0} \aps{p_i^{(n)} - p_i^{(\infty)}} \le n_0 \cdot \norm[1]{P^{(n)} - P^{(\infty)}}_{\infty} < n_0 \cdot \frac{\varepsilon}{n_0} = \varepsilon.
	\end{equation}
	On the other hand, we have that for all $n \ge n_1$:
	\begin{equation}\label{eq:tail_of_p^n}
		\sum_{i = n_0 + 1}^{\infty} p_i^{(n)} < \varepsilon + \sum_{i = 1}^{n_0} (p_i^{(\infty)} - p_i^{(n)}) \le \varepsilon + \sum_{i = 1}^{n_0} \aps{p_i^{(\infty)} - p_i^{(n)}} < 2\varepsilon,
	\end{equation}
	where in the first inequality we used equation \eqref{eq:tail_of_p^inf} together with the fact that $\sum_{i = 1}^{\infty} p_i^{(n)} = 1$ and in the last inequality we used equation \eqref{eq:start_of_l^1_bound}. Combining Equations \eqref{eq:tail_of_p^n} and \eqref{eq:tail_of_p^inf} we get
\begin{equation}\label{eq:tail_of_l^1_bound}
		\sum_{i = n_0 + 1}^{\infty} \aps{p_i^{(n)} - p_i^{(\infty)}} \le \sum_{i = n_0 + 1}^{\infty} p_i^{(n)} + \sum_{i = n_0 + 1}^{\infty} p_i^{(\infty)} < 2\varepsilon + \varepsilon = 3\varepsilon,
	\end{equation}
	for all $n \ge n_1$. Finally, we have
	\begin{equation*}
		\norm[1]{P^{(n)} - P^{(\infty)}}_1 = \sum_{i = 1}^{n_0} \aps{p_i^{(n)} - p_i^{(\infty)}} + \sum_{i = n_0 + 1}^{\infty} \aps{p_i^{(n)} - p_i^{(\infty)}} < \varepsilon + 3\varepsilon = 4\varepsilon,
	\end{equation*}
	where we used equations \eqref{eq:start_of_l^1_bound} and \eqref{eq:tail_of_l^1_bound}. This proves that $\norm[1]{P^{(n)} - P^{(\infty)}}_1 \xrightarrow[n \rightarrow \infty]{} 0$ assuming that $\norm[1]{P^{(n)} - P^{(\infty)}}_{\infty} \xrightarrow[n \rightarrow \infty]{} 0$, which is exactly what we wanted to prove.
\end{proof}

\subsection{Auxiliary results for Section \ref{sec:Asymptotic_fairness}}

\begin{PROP}\label{prop:log_ineq}
	Let $x, y > 0$ such that $x + y < 1$. Then
	\begin{equation}\label{eq:log_ineq}
		1 + \frac{\log(1 - x)}{x} + \frac{\log(1 - y)}{y} > \frac{\log(1 - (x + y))}{x + y}.
	\end{equation}
\end{PROP}
\begin{proof}
	Define
	\begin{equation*}
		g(x, y) := 1 + \frac{\log(1 - x)}{x} + \frac{\log(1 - y)}{y} - \frac{\log(1 - (x + y))}{x + y}.
	\end{equation*}
	We now show that $g(x, y) > 0$ for all $x, y > 0$ such that $x + y < 1$. Let $y \in (0, 1)$ be arbitrary but fixed. Notice  that
	\begin{align*}
		\lim_{x \to 0} g(x, y)
		& = 0.
	\end{align*}
	Hence, to prove that $g(x, y) > 0$ for all $x \in (0, 1 - y)$ (for fixed $y$) it is sufficient to show that $x \mapsto g(x, y)$ is strictly increasing on $(0, 1 - y)$.  We have
	\begin{align*}
		\frac{\partial g}{\partial x} (x, y)
		& = \frac{\log(1 - (x + y))}{(x + y)^2} + \frac{1}{(x + y)(1 - (x + y))} - \frac{\log(1 - x)}{x^2} - \frac{1}{x(1 - x)} \\
		& = h(x + y) - h(x)
	\end{align*}
	for
	\begin{equation*}
		h(x) := \frac{\log(1 - x)}{x^2} + \frac{1}{x(1 - x)}.
	\end{equation*}
	Therefore, it is enough to show that $h(x)$ is a strictly increasing function on $(0, 1)$ since then (for $y \in (0, 1)$ and $x \in (0, 1 - y)$) we would have $\frac{\partial g}{\partial x}(x, y) = h(x + y) - h(x) > 0$. We verify that $h(x)$ is strictly increasing on $(0, 1)$ by showing that $h'(x) > 0$ on $(0, 1)$. We have that
	\begin{equation*}
		h'(x) = \frac{1}{x^3} \OBL{\frac{x(3x - 2)}{(1 - x)^2} - 2\log(1 - x)}.
	\end{equation*}
	Hence, it remains to prove that	\begin{equation}\label{eq:key_ineq}
		\log(1 - x) < \frac{x(3x - 2)}{2(1 - x)^2}.
	\end{equation}
One way to see this is to prove that	
\begin{equation}
\log(1 - x) < -x - \frac{x^2}{2} < \frac{x(3x - 2)}{2(1 - x)^2}.
\end{equation}
As this is  basic analysis we omit the details.
\end{proof}
\begin{LM}\label{lem:Lagrange_multipliers}
	Let $p \in (0, 1)$ and let $D = \{x_1, x_2, \ldots, x_r \in (0, 1) : \sum_{j = 1}^r x_{j} = 1\}$. The function $\FJADEF{g}{D}{\bbR}$ defined by
	\begin{equation*}
		g(x_1, x_2, \ldots, x_r) = \sum_{j = 1}^r \frac{\log(1 - p x_j)}{p x_j}
	\end{equation*}
	has a unique maximum on the set $D$ at the point $(x_1, x_2, \ldots, x_r) = (\frac{1}{r}, \frac{1}{r}, \ldots, \frac{1}{r})$.
\end{LM}
The proof of Lemma \ref{lem:Lagrange_multipliers}  is a  standard application of Lagranges's multiplier and omitted.

\begin{PROP}\label{prop:tau_r_is_inc}
	Let $p \in (0, 1)$ and let
	\begin{equation*}
		\tau_r(p) = (1 - p) \UGL{r + \frac{r^2 \log(1 - \frac{p}{r})}{p} - \frac{\log(1 - p)}{p} - 1}.
	\end{equation*}
Then, the sequence $(\tau_r(p))_{r \in \bbN}$ is an increasing sequence for all $p \in (0, 1)$ and it holds that
	\begin{equation*}
		\tau(p) = \lim_{r \to \infty} \tau_r(p) = (1 - p) \OBL{-\frac{p}{2} - \frac{\log(1 - p)}{p} - 1}.
	\end{equation*}
\end{PROP}
\begin{proof}
	Let us first show that the sequence $(\tau_r(p))_{r \in \bbN}$ is strictly increasing. For this, it is sufficient to show that for $p\in (0, 1)$ the function
	\begin{equation*}
		\phi(x) := x + \frac{x^2 \log(1 - \frac{p}{x})}{p}
	\end{equation*}
	is strictly increasing on $[1, \infty)$, because the sequence $(\tau_r(p))_{r \in \bbN}$ satisfies $\tau_r(p) = (1 - p) (\phi(r) - \phi(1))$. We will show that $\phi'(x) > 0$ for all $x \in [1, \infty)$. We have
	\begin{equation*}
		\phi'(x) = 1 + \frac{2x\log(1 - \frac{p}{x})}{p} + \frac{x}{x - p}.
	\end{equation*}
	Since $\lim_{x \to \infty} \phi'(x) = 0$ and $\phi'(x)$ is a continuous function on $[1, \infty)$, it is now enough to show that $\phi'(x)$ is strictly decreasing on $[1, \infty)$ to be able to conclude that $\phi'(x) > 0$ for all $x \in [1, \infty)$. Observe that
	\begin{equation*}
		\phi''(x) = \frac{2}{p} \log\OBL{1 - \frac{p}{x}} + \frac{2x - 3p}{(x - p)^2}.
	\end{equation*}
We can now check that $\phi''(x) < 0$ on $[1, \infty)$:
	\begin{align*}
		\phi''(x) < 0
		& \Leftrightarrow \frac{2}{p} \log\OBL{1 - \frac{p}{x}} < \frac{3p - 2x}{(x - p)^2} \\
		& \Leftrightarrow \log\OBL{1 - \frac{p}{x}} < \frac{\frac{p}{x} \OBL{3 \frac{p}{x} - 2}}{2 \OBL{1 - \frac{p}{x}}^2}.
	\end{align*}
	Since $p \in (0, 1)$ and $x \in [1, \infty)$, we have $p/x \in (0, 1)$ so the desired inequality follows from \eqref{eq:key_ineq}. Hence,   $\phi'(x)$ is decreasing and we have that $\phi'(x) > 0$ on $[1, \infty)$, which is exactly what we wanted to prove. 

	Let us now calculate the limit of the sequence $(\tau_r(p))_{r \in \bbN}$. Notice that	\begin{equation}\label{eq:tau_r_p_for_LH}
		\tau_r(p) = (1 - p) \UGL{\frac{1 + \frac{\log(1 - \frac{p}{r})}{\frac{p}{r}}}{\frac{1}{r}} - \frac{\log(1 - p)}{p} - 1}.
	\end{equation}
	Applying L'Hospital's rule twice, we obtain
	\begin{equation*}
		\lim_{x \to 0^+} \frac{1 + \frac{\log(1 - px)}{px}}{x} = \lim_{x \to 0^+} \frac{\frac{-p^2 x}{1 - px} - p \log(1 - px)}{p^2 x^2} = \lim_{x \to 0^+} \frac{-\frac{p^3 x}{(1 - px)^2}}{2p^2 x} = -\frac{p}{2}.
	\end{equation*}
	Plugging this in Equation \eqref{eq:tau_r_p_for_LH} we obtain
	\begin{equation*}
		\tau(p) = \lim_{r \to \infty} \tau_r(p) = (1 - p) \OBL{-\frac{p}{2} - \frac{\log(1 - p)}{p} - 1},
	\end{equation*}
	which concludes the proof.
\end{proof}

\section*{Acknowledgement} 
We wish to thank Serguei Popov for suggesting the name ``greedy sampling'' and the whole IOTA research team for stimulating discussions.

A.\ Gutierrez was supported by the \textit{Austrian Science Fund} (FWF) under project P29355-N35. S.\ \v Sebek was supported by the \textit{Austrian Science Fund} (FWF) under project P31889-N35 and \textit{Croatian Science Foundation} under project 4197. These financial supports are gratefully acknowledged.

\bibliographystyle{abbrv}
\bibliography{AF}

\end{document}